\documentclass[12pt]{amsart}
\usepackage{amssymb,amsmath,amsfonts,latexsym}
\usepackage{bm, enumerate}
\usepackage{mathtools}
\usepackage{xcolor}
\newcommand{\newsection}[1]{\setcounter{equation}{0} \section{#1}}
\newcommand{\bea}{\begin{eqnarray}}
\newcommand{\eea}{\end{eqnarray}}

\newcommand{\cla}{\mathcal{A}}
\newcommand{\clb}{\mathcal{B}}
\newcommand{\clc}{\mathcal{C}}
\newcommand{\cld}{\mathcal{D}}
\newcommand{\cle}{\mathcal{E}}
\newcommand{\clf}{\mathcal{F}}

\newcommand{\clh}{\mathcal{H}}
\newcommand{\clk}{\mathcal{K}}
\newcommand{\cll}{\mathcal{L}}
\newcommand{\clm}{\mathcal{M}}
\newcommand{\cln}{\mathcal{N}}
\newcommand{\clo}{\mathcal{O}}

\newcommand{\clq}{\mathcal{Q}}
\newcommand{\clr}{\mathcal{R}}
\newcommand{\cls}{\mathcal{S}}
\newcommand{\clt}{\mathcal{T}}

\newcommand\bH{\mathbb{H}}

\newcommand{\B}{\mathbb{B}}
\newcommand{\C}{\mathbb{C}}
\newcommand{\D}{\mathbb{D}}
\newcommand{\N}{\mathbb{N}}
\newcommand{\Z}{\mathbb{Z}}
\newcommand{\z}{\bm{z}}
\newcommand{\w}{\bm{w}}

\def \qed {\hfill \vrule height6pt width 6pt depth 0pt}
\def\textmatrix#1&#2\\#3&#4\\{\bigl({#1 \atop #3}\ {#2 \atop #4}\bigr)}
\def\dispmatrix#1&#2\\#3&#4\\{\left({#1 \atop #3}\ {#2 \atop #4}\right)}
\newcommand{\be}{\begin{equation}}
\newcommand{\ee}{\end{equation}}
\newcommand{\ben}{\begin{eqnarray*}}
\newcommand{\een}{\end{eqnarray*}}

\newcommand{\bi}{\begin{itemize}}
\newcommand{\ei}{\end{itemize}}

\newtheorem{Theorem}{\sc Theorem}[section]
\newtheorem{Lemma}[Theorem]{\sc Lemma}
\newtheorem{Proposition}[Theorem]{\sc Proposition}
\newtheorem{Corollary}[Theorem]{\sc Corollary}
\newtheorem{Definition}[Theorem]{\sc Definition}
\newtheorem{Example}[Theorem]{\sc Example}
\newtheorem{Remark}[Theorem]{\sc Remark}

\newtheorem{Remarks}[Theorem]{\sc Remarks}
\newtheorem{Note}[Theorem]{\sc Note}
\newtheorem{Question}{\sc Question}
\newtheorem{ass}[Theorem]{\sc Assumption}
\newcommand{\bt}{\begin{Theorem}}
\def\beginlem{\begin{Lemma}}
\def\beginprop{\begin{Proposition}}
\def\begincor{\begin{Corollary}}
\def\begindef{\begin{Definition}}
\def\beginexamp{\begin{Example}}
\def\beginrem{\begin{Remark}}
\def\beginq{\begin{Question}}
\def\beginass{\begin{ass}}
\def\beginnote{\begin{Note}}
\newcommand{\et}{\end{Theorem}}
\def\endlem{\end{Lemma}}
\def\endprop{\end{Proposition}}
\def\endcor{\end{Corollary}}
\def\enddef{\end{Definition}}
\def\endexamp{\end{Example}}
\def\endrem{\end{Remark}}
\def\endq{\end{Question}}
\def\endass{\end{ass}}
\def\endnote{\end{Note}}

\begin{document}

\title[Commutant lifting]{Commutant lifting in several variables }

\author[Barik]{Sibaprasad Barik}
\address{Department of Mathematics, Indian Institute of Technology Bombay, Powai, Mumbai, 400076, 
India}
\email{sbank@math.iitb.ac.in, sibaprasadbarik00@gmail.com}

\author[Bhattacharjee]{Monojit Bhattacharjee}
\address{Department of Mathematics, Indian Institute of Technology Bombay, Powai, Mumbai, 400076,
India}
\email{mono@math.iitb.ac.in, monojit.hcu@gmail.com}

\author[Das] {B. Krishna Das}
\address{Department of Mathematics, Indian Institute of Technology Bombay, Powai, Mumbai, 400076,
India}
\email{dasb@math.iitb.ac.in, bata436@gmail.com}



\begin{abstract}
In this article we study commutant lifting, more generally intertwining 
lifting, for different reproducing kernel Hilbert spaces over two domains in $\C^n$, namely the unit ball and the unit polydisc. The reproducing kernel Hilbert spaces we consider are mainly weighted Bergman spaces. Our commutant lifting results are explicit in nature and that is why these results are new even in one variable $(n=1)$ set up.

\end{abstract}

\maketitle

\section*{Notation}

\begin{list}{\quad}{}
\item $\mathbb N$\quad \quad \quad\quad The set of all natural numbers.
\item $\mathbb Z_{+}$\quad \quad \quad\ The set of all positive integers.
\item $\mathbb{D}$ \quad \quad \quad\ \ Open unit disc in the complex plane
$\mathbb{C}$.

\item  $\mathbb{D}^n$ \; \quad \quad\  \ Open unit polydisc in $\mathbb{C}^n$.

\item $\B^n$\; \quad\quad\ \  \ Open unit ball in $\C^n$.

\item  $\clh$, $\cle$ \; \quad\ \ Hilbert spaces.

\item $\clb(\clh)$ \quad \;\ \ The space of all bounded linear
operators on $\clh$.

\item $H^2_{\cle}(\mathbb{D})$ \quad \ \  $\cle$-valued Hardy space on
$\mathbb{D}$.



\end{list}

All Hilbert spaces are assumed to be over the complex numbers.

\newsection{Introduction} \label{Introduction}

One of the well-studied problems in function theory is the classical interpolation problem of bounded analytic functions: the Carathedory-Fejer interpolation problem and the Nevanlinna-Pick interpolation problem. The classical Nevanlinna-Pick interpolation problem (\cite{Pick, N}) over the unit disk ($\D$) asks the following: Given any set of $n$ distinct points $\{z_i\}_1^n \subseteq \D$ and arbitrary $n$ points $\{w_i\}_1^n \subseteq \D$ whether there exists a bounded holomorphic 
function $\Phi$ on $\D$ with $\|\Phi\|_{\infty}\le 1$ such that $\Phi(z_i)=w_i$ for all $i=1,\dots,n$.  It was Sarason (\cite{Sar}) who first observed that there is a natural operator theoretic connection, by means of a commutant lifting theorem, to this function theoretic problem. The 
key observation was that having such a bounded analytic function $\Phi$
is same as having a lift (co-extension) of a certain operator defined on a certain co-invariant
subspace of $H^2(\D)$, the Hardy space over $\D$.  In this context, Sarason proved that if $\clq$ is a co-invariant subspace of $H^2(\D)$ and if $X\in\clb(\clq)$ is a contraction such that $X(P_{\clq}M_z|_{\clq})=(P_{\clq}M_z|_{\clq})X$ then there exists a $\Phi$
in $H^{\infty}(\D)$, the algebra of all bounded holomorphic functions on $\D$, 
such that $M_{\Phi}^*|_{\clq}=X^*$ and $\|\Phi\|_{\infty}\le 1$ where 
$M_z$ is the shift on $H^2(\D)$ and $M_{\Phi}$ is the multiplication operator on 
$H^2(\D)$ induced by the bounded analytic function $\Phi$. Adding more operator theoretic flavor, Sz.-Nagy and Foias generalized this for arbitrary contractions and prove the following intertwining lifting result (also see  ~\cite{DMP} for a matricial approach): If $T$ and $S$ are contractions on $\clh$ and $\clk$
with minimal isometric dilations (\cite{SF}) $U\in \clb(\tilde{\clh})$ and $V\in\clb(\tilde{\clk})$ respectively, and if $X\in\clb(\clh,\clk)$ is a contraction such that 
$XT=SX$ then there exists a contraction $Y\in \clb(\tilde{\clh},\tilde{\clk})$ satisfying $YU=VY$ and $Y^*|_{\clh}=X^*$. Apart from its function theoretic applications and interesting operator theoretic consequences, intertwining lifting theorem has applications in the control theory (see ~\cite{FF}).
There have been significant amount of research devoted to finding 
multivariate analogues of this intertwining lifting theorem and its counterpart in the setting of reproducing kernel Hilbert spaces over different domains in $\C^n$ with application to interpolations. An incomplete list of references is ~\cite{Ag, AM, AT, BLTT, BT, BTV, DPST, EP, Mu, N, Pick}.   

The aim of this article is to study intertwining lifting theorem for different reproducing kernel Hilbert spaces on two domains in $\C^n$, namely the unit ball $\B^n$ and the unit polydisc $\D^n$. In the case of $\B^n$, let $\bH_m(\B^n,\cle)$ $(m\in\mathbb N)$ be the $\cle$-valued  
weighted Bergman space with kernel 
\[
K_m(\z,\w)=\frac{1}{(1-\langle \z,\w\rangle)^m}I_{\cle}, \quad (\z,\w\in\B^n)
\]
for some coefficient Hilbert space $\cle$. A closed subspace $\clq$ of $\bH_m(\B^n,\cle)$ is a co-invariant subspace if $\clq$ is jointly invariant under
the adjoint of the $n$-tuple of shifts $(M_{z_1},\dots, M_{z_n})$ on $\bH_m(\B^n,\cle)$. One of the problems that we consider in this setting is the following.

\textit{
If $\clq_1\subseteq \bH_m(\B^n,\cle)$ and $\clq_2\subseteq \bH_m(\B^n,\cle_*)$ are co-invariant subspaces for some Hilbert spaces $\cle$ and $\cle_*$ and if $X\in \clb(\clq_1,\clq_2)$ is a contraction such that 
\begin{equation}\label{intertwining}
X(P_{\clq_1}M_{z_i}|_{\clq_1})=(P_{\clq_2}M_{z_i}|_{\clq_2})X
\end{equation}  
for all $i=1,\dots,n$ then does there exists a Schur-Agler function (defined below) $\Phi: \B^n\to \clb(\cle,\cle_*)$ such that $X^*=M_{\Phi}^*|_{\clq_2}$?}

\noindent It is because of the fact that Schur-Agler class of functions is more tractable in the sense of its transfer function realization
 (see Theorem~\ref{char Schur-Agler class in ball}  below),
 we seek liftings corresponding to multipliers in this class. Also note that for a Schur-Agler function $\Phi: \B^n\to \clb(\cle,\cle_*)$,
the corresponding multiplication operator $M_{\Phi}: \bH_m(\B^n,\cle)\to \bH_m(\B^n,\cle_*)$ defined by $f\mapsto \Phi f$ is a bounded operator
 (see Proposition~\ref{inclusion}  below).
It turns out that it is not possible to find a lift of $X$ in the Schur-Agler class multipliers, in general. We find a necessary
and sufficient condition on $X$ for which it has a lift
 in the Schur-Agler class. 
 The necessary and sufficient condition is simply that 
\[
(1-\sigma_S)^{m-1}(I-XX^*)\ge 0
\]
where $S=(P_{\clq_2}M_{z_1}|_{\clq_2},\dots, P_{\clq_2}M_{z_n}|_{\clq_2})$ and 
$\sigma_S:\clb(\clq_2)\to\clb(\clq_2)$ is defined by 
\[Y\mapsto \sum_{i=1}^nS_iYS_i^*.\]
In the case when $X$ satisfies the above the positivity condition, we also 
obtain an explicit description of its lift $\Phi$ by means of finding a unitary whose transfer function is $\Phi$. 
We should mention here that the present work is based on ideas found in the context of sharp von Neumann inequality on distinguished varieties in ~\cite{DS}.
For $m=1$, the Hilbert space
 $\bH_1(\B^n,\cle)$ is known as the Drury-Arveson space and the above positivity condition is satisfied by any contraction $X$, in this case. Therefore, in the case of Drury-Arveson space any contraction satisfying the intertwining relation ~\eqref{intertwining} can be lifted to a multiplier corresponding to a Schur-Agler function. 
This result is known (see ~ \cite{AT, BTV, DL, EP}). However, our proof is not only new but 
also provides an explicit description of the lifting. The explicitness of the lifting is naturally important from the point of view of its application in interpolation. As an immediate consequence, this also provides an alternate proof of transfer function realization of Schur-Agler functions on $\B^n$ which was proved earlier in ~\cite{EP} and also in ~\cite{BTV}. 
We continue our study to consider the above intertwining lifting problem for the case when 
$\clq_1$ is a co-invariant subspace of $\bH_m(\B^n,\cle)$ and $\clq_2$ is a 
co-invariant subspace of $\bH_p(\B^n,\cle_*)$ for $p>m$ (see Theorem~\ref{ILT for p>m}) and which led to factorization of certain type of multipliers (see Theorem ~\ref{factorization}). 
A particular case, namely $m=1$ and $p>1$, is studied recently in
 ~\cite{DPST}.  

Several variable analogue of Sz.-Nagy and Foias intertwining lifting theorem is rather complicated and fails in general (see \cite{FF, Mu}). In particular, Muller~\cite{Mu} showed that if $(T_1, T_2)$ is
a pair of commuting contractions on $\clh$
with minimal regular dilation $ (U_1, U_2)$ on $\clk$  (\cite{SF}) and if
$X \in \clb(\clh)$ commutes with $T_1$ and $T_2$, 
then there does not exist any operator $Y \in \clb(\clk)$ which 
commutes with  $U_1$ and $ U_2$ such that $X = P_{\clh}Y |_{\clh}$. 
On the other hand, in the setting of $\cle$-valued Hardy space over the unit polydisc $H^2_{\cle}(\D^n)$ $(n \ge 2)$ Ball, Li, Timotin and Trent (\cite{BLTT}) found a necessary and sufficient condition for intertwining lifting in the Schur-Agler class of functions on $\D^n$. In this article, we generalize this result in the setting of weighted Bergman spaces over $\D^n$. To be more precise,  corresponding to each $\bm\gamma=(\gamma_1,\dots,\gamma_n)\in \N^n$, let $A^2_{\bm\gamma}(\D^n,\cle)$ be the $\cle$-valued weighted Bergman space over $\D^n$ with kernel 
 \[
 K_{\bm\gamma}(\z,\w)=\prod_{i=1}^n (1-z_i\bar{w_i})^{-\gamma_i}I_{\cle}
 \quad (\z,\w\in \D^n),
 \] 
 where $\cle$ is a Hilbert space. If $\clq_i$ is a co-invariant subspace 
 of $A^2_{\bm\gamma}(\D^n,\cle_i)$ for $i=1,2$ and if $X\in\clb(\clq_1,\clq_2)$ is a contraction such that for all $i=1,\dots,n$,
 \[
X (P_{\clq_1}M_{z_i}|_{\clq_1})= (P_{\clq_2}M_{z_i}|_{\clq_2})X,
 \]
then we find a necessary and sufficient condition on $X$ for which there is a Schur-Agler function (defined below) $\Phi:\D^n\to \clb(\cle_1,\cle_2)$ such that $X^*=M_{\Phi}^*|_{\clq_2}$  (see Theorem~\ref{Main Theorem in D^n} below for more details). In the case when $X$ satisfies the 
necessary and sufficient condition, we find such a Schur-Agler function explicitly.
 This 
uses an appropriate modification of techniques found in ~\cite{BLTT}.    
 
The plan of the paper is as follows. In the next section, we define reproducing kernels on $\B^n$ and $\D^n$ which we consider and describe few properties of their multiplier algebras. In Section~\ref{Commutant lifting in B^n}, we study intertwining lifting theorem for reproducing kernels on $\B^n$. The intertwining 
lifting theorem for kernels on $\D^n$ is considered in Section ~\ref{Commutant lifting in D^n}. We conclude the paper with some examples and remarks in Section~\ref{Concluding section}.

\newsection{Preliminaries} \label{Preliminaries}

For an arbitrary set $\Lambda$, an operator-valued function
$K: \Lambda \times \Lambda \to \clb(\clh)$ is said to be \textit{positive definite} if 
$ \sum_{i,j}^n \langle K(\lambda_i, \lambda_j)\eta_i, \eta_j \rangle \geq 0$ for every choice of $\lambda_1,\ldots,\lambda_n \in \Lambda $
and $\eta_1,\ldots, \eta_n \in \clh$. 
A renowned theorem of Kolmogorov and Aronszajn (Theorem I.5.1 in \cite{MP}) completely characterizes all positive definite functions. It says that a function $K: \Lambda \times \Lambda \to \clb(\clh)$ is positive definite if and only if there exists a Hilbert space $\clk$ and a function
$F: \Lambda \to \clb(\clk, \clh)$ such that 
$K(z,w)= F(z)F(w)^*$ for all $z,w \in \Lambda$. In this article, we mainly deal with kernel functions
on two different  domains in $\C^n$,
namely the polydisc ($\D^n$) and the unit ball ($\B^n$). 
Multiplier algebras of reproducing kernel Hilbert spaces corresponding 
to such kernel functions are inevitable in the present consideration.
A brief descriptions of these are considered in the subsections below.
For Hilbert spaces $\cle_1$ and $\cle_2$ and for any domain
$\Lambda \subseteq \C^n$, we denote by $H^{\infty}(\Lambda, \clb(\cle_1, \cle_2))$ the Banach algebra of all bounded analytic functions $\phi : \Lambda \to \clb(\cle_1, \cle_2)$ with respect to the norm $\|\phi\|_{\infty} = \text{sup}\, \{ \|\phi(z)\| : z \in \Lambda \}$ and by $H^{\infty}_1(\Lambda, \clb(\cle_1, \cle_2))$ we denote the closed unit ball of $H^{\infty}(\Lambda, \clb(\cle_1, \cle_2))$. 


\subsection{Bergman Spaces over the unit ball $\B^n$:}

For every $m\in\N$, the positive definite function
$K_m: \B^n \times \B^n \to \C$ defined by 
\[ K_m(\bm z, \bm w) = (1 - \langle \bm z,\bm w \rangle)^{-m} = ( 1 - \sum_{i=1}^n z_i \bar{w}_i )^{-m}, \quad \quad (\bm z,\bm w \in \B^n). \] 
is known as the kernel for weighted Bergman space over 
$\B^n$. 
For an arbitrary Hilbert space $\cle$, we denote by
$\mathbb{H}_m(\B^n, \cle)$ the $\cle$-valued weighted Bergman space over $\B^n$ with kernel 
\[ (\bm z,\bm w) \to K_m(\bm z,\bm w)I_{\cle}\quad 
(\bm z,\bm w \in \B^n).\]
  Note that, for each $\bm z\in \B^n$, we have (cf. page 983, \cite{MV}) 
\begin{equation} \label{rhom}
(1 - \sum_{i=1}^n z_i)^{-m} = \sum_{\bm k \in \Z_+^n} \rho_m(\bm k){\bm z}^{\bm k},
\end{equation}
where $\rho_m(\bm k) = \frac{(m+ |\bm k| - 1)!}{\bm k ! (m-1)!}$ and 
${\bm z}^{\bm k} = z_1^{k_1}\ldots z_n^{k_n}$ for all $\bm k \in \Z_+^n$. Using this, one can represent the Hilbert function space
$\mathbb{H}_m(\B^n, \cle)$ in the following concrete way
(see \cite{Ar} and \cite{MV}) 
\[ \mathbb{H}_m(\B^n, \cle) = \{ f \in \sum_{\bm k \in \Z_+^n} a_{\bm k}{\bm z}^{\bm k} \in \clo(\B^n, \cle): \|f\|^2:= \sum_{\bm k \in \Z_+^n} \frac{\|a_{\bm k}\|^2_{\cle}}{\rho_m(\bm k)} < \infty \}. \]
For $m=1,n$ and $n+1$, the corresponding Hilbert spaces
$\mathbb{H}_1(\B^n, \cle), 
\mathbb{H}_n(\B^n, \cle)$ and $\mathbb{H}_{n+1}(\B^n, \cle)$
are known as the $\cle$-valued Drury-Arveson space, the $\cle$-valued Hardy space
and the $\cle$-valued Bergman space over $\B^n$, respectively.
Following standard notation, we denote the Drury-Arveson space $\mathbb{H}_1(\B^n, \cle)$ by $H^2_n(\cle)$. 
The commuting $n$-tuple of co-ordinate multiplication operators $(M_{z_1},\dots,M_{z_n})$ on $\bH_m(\B^n,\cle)$ are called shifts on $\bH_m(\B^n,\cle)$. These shifts on weighted Bergman spaces are models of a class of commuting $n$-tuple of operators which we define next. 
A commuting $n$-tuple of commuting contractions $T=(T_1,\dots,T_n)$ 
on $\clh$ is said to be an \textit{$m$-hypercontraction} if 
\[
(1-\sigma_T)^i(I_{\clh})\ge 0
\]
for $i=1,m$ where $\sigma_T: \clb(\clh)\to \clb(\clh)$ is a completely positive map 
defined by 
\begin{equation}\label{CP}
\sigma_T(X)=\sum_{i=1}^n T_iXT_i^* \quad (X\in\clb(\clh)).
\end{equation}
It turns out that the above positivity for $i=1,m$ is equivalent to the positivity for all $i=1,2,\dots,m$ (see ~\cite[Lemma 2]{MV}). In other words if $T$ is an $m$-hypercontraction then it is also $p$-hypercontraction for all $p=1,\dots,m$. 
An $m$-hypercontraction $T$ on $\clh$ is said to be \textit{pure} if 
$\sigma_T^j(I_{\clh}) \to 0$ in strong operator topology as $j \to \infty$.
The \textit{defect operator} and the \textit{defect space} of $T$  is denoted by $D_{m,T^*}$ and $\cld_{m,T^*}$ respectively and defined by 
\[ D_{m,T^*} = [(1- \sigma_T)^m(I_{\clh})]^{\frac{1}{2}} \ \text{ and }
\cld_{m,T^*} = \overline{\text{ran}}D_{m,T^*}.
 \]
If $T$ is a pure $m$-hypercontraction, then the  \textit{canonical dilation map}
$\pi_T : \clh \to \bH_m(\B^n, \cld_{m,T^*})$, defined by 
\begin{equation}\label{dilation map for m-hyper in ball set up} 
(\pi_T h)(\bm z) 
 = \sum_{\bm k \in \Z_+^n}\rho_m(\bm k)(D_{m,T^*} T^{*\bm k}h) {\bm z}^{\bm k}  \quad \quad (h \in\clh, \bm z \in \B^n)
\end{equation}
is an isometry and 
\[ \pi_T T_i^* = M_{z_i}^* \pi_T \quad \quad (i=1,\ldots,n),\]
where $\rho_m(\bm k)$ is as in ~\eqref{rhom}. In other words, 
\[
T\cong (P_{\clq}M_{z_1}|_{\clq},\dots, P_{\clq}M_{z_n}|_{\clq}),
\]
where $\clq=\mbox{ran} \pi_T$ is a co-invariant subspace of
 $\bH_m(\B^n,\cld_{m,T^*})$. This well-known models of pure $m$-hypercontractions is obtained in ~\cite{MV}.

For any two arbitrary Hilbert spaces $\cle_1$ and $\cle_2$ and $m,p\in\N$, a holomorphic function $\Phi: \B^n \to \clb(\cle_1,\cle_2)$
is said to be a \textit{multiplier} from
$\mathbb{H}_m(\B^n, \cle_1)$ to $\mathbb{H}_p(\B^n, \cle_2)$ if 
\[ \Phi f \in \mathbb{H}_p(\B^n, \cle_2) \ \text{for all}\,\, f \in \mathbb{H}_m(\B^n, \cle_1) .\]
We denote by $\clm(\mathbb{H}_m(\B^n, \cle_1), \mathbb{H}_p(\B^n, \cle_2) )$ the space of all multipliers from $\mathbb{H}_m(\B^n, \cle_1)$ to $\mathbb{H}_p(\B^n, \cle_2)$. 
By an immediate consequence of closed graph theorem, 
each member $\Phi$ of  $ \clm(\mathbb{H}_m(\B^n, \cle_1), \mathbb{H}_p(\B^n, \cle_2) )$ induces a bounded linear operator 
\[ M_{\Phi}: \mathbb{H}_m(\B^n, \cle_1) \to \mathbb{H}_p(\B^n, \cle_2), \quad \quad f \mapsto \Phi f, \] 
and it is known as the multiplication operator induced by $\Phi$. 
With the induced norm, that is 
\[ \|\Phi\|: = \|M_{\Phi}\|,\]

$\clm(\mathbb{H}_m(\B^n, \cle_1), \mathbb{H}_p(\B^n, \cle_2) ) \subseteq \mathcal{O}( \B^n, \clb(\cle_1, \cle_2) )$ becomes a Banach space. The unit ball of $\clm(\mathbb{H}_m(\B^n, \cle_1), \mathbb{H}_p(\B^n, \cle_2) )$
is denoted by $\clm_1(\mathbb{H}_m(\B^n, \cle_1), \mathbb{H}_p(\B^n, \cle_2) )$ and elements of the unit ball are called contractive multipliers.
It is worth mentioning the following recent characterization of multipliers on weighted Bergman spaces. Although the characterization is valid for large class of reproducing kernel Hilbert spaces but 
we sate it only for weighted Bergman spaces.
 \begin{Theorem}[cf. \cite{JS}] 
 \label{mul-char-ball}
	Let $X \in \clb(\bH_m(\B^n, \cle_1), \bH_p(\B^n, \cle_2))$. Then 
	\[ XM_{z_i}  = M_{z_i} X, \quad \quad (i=1,\ldots,n) \]
 if and only if there exits $\Theta \in  \clm(\mathbb{H}_m(\B^n, \cle_1), \mathbb{H}_p(\B^n, \cle_2) ) $ such that 
	$X=M_{\Theta}$. 
Here $M_{z_i}$ in the left side of the above
identity is the shift on $\bH_m(\B^n,\cle_1)$ and $M_{z_i}$ in the right side is the shift on $\bH_p(\B^n,\cle_2)$.
\end{Theorem}

 In ~\cite{Ar}, Arveson shows that the space of multipliers from $\mathbb{H}_p(\B^n, \cle_1)$ to $\mathbb{H}_p(\B^n, \cle_2)$ is strictly contained in $H^{\infty}(\B^n, \clb(\cle_1, \cle_2))$, that is, 
\[ \clm(\mathbb{H}_p(\B^n, \cle_1), \mathbb{H}_p(\B^n, \cle_2) )  \subsetneq H^{\infty}(\B^n, \clb(\cle_1, \cle_2)) .\] 
 In particular, for $p=1$, the space of Drury-Arveson space
 multipliers $\clm(H^2_n(\cle_1), H^2_n(\cle_2) )$ is strictly contained in $H^{\infty}(\B^n, \clb(\cle_1, \cle_2))$. 
 The set of all contractive multipliers form Drury-Arveson space to itself, that is, 
$\clm_1(H^2_n(\cle_1), H^2_n(\cle_2) )$ is known as \textit{Schur-Agler} class of functions on $\B^n$. The Schur-Agler class of functions 
has a well-known characterization in terms of transfer functions
(cf. \cite{EP,BTV}) as follows. 
\begin{Theorem} \label{char Schur-Agler class in ball}
Suppose $\phi: \B^n \to \clb(\cle_1,\cle_2)$ is a holomorphic function. Then the function $\phi \in \clm_1(H^2_n(\cle_1), H^2_n(\cle_2) )$ if and only if there exists a Hilbert space $\clk$ and a unitary 
\[ U= \begin{bmatrix}
A & B \cr
C & D
\end{bmatrix}: \cle_1 \oplus \clk \to \cle_2 \oplus \clk^n \]
such that $\phi(\bm z) = A + C(I -ZD)^{-1}ZB $ for all $\bm z \in \B^n$, where the row contraction $Z \in \clb( \clk^n, \clk)$ is defined by $Z(h_1,\ldots,h_n) = \sum_{i=1}^n z_ih_i$ $((z_1,\ldots,z_n) \in \B^n)$. 
\end{Theorem}

The connection between different multipliers spaces, in particular Schur-Agler class and $\clm_1(\mathbb{H}_p(\B^n, \cle_1), \mathbb{H}_p(\B^n, \cle_2) )$ with $p \geq 2$, plays an important role in this article. It is highlighted in a recent article for general reproducing kernel Hilbert space (see Theorem 4.1 in \cite{CH}) but for our purpose we state a special case of it.
\begin{Proposition}\label{inclusion}
The Schur-Agler class is contained in the class of contractive multipliers of weighted Bergman spaces, that is, for all $p \geq 2$,
	\[ \clm_1(H^2_n(\cle_1), H^2_n(\cle_2) ) \subseteq \clm_1(\mathbb{H}_p(\B^n, \cle_1), \mathbb{H}_p(\B^n, \cle_2) ).\]  
\end{Proposition}

\subsection{Bergman spaces over the polydisc $\D^n$:}

For each $\bm {\gamma}=( {\gamma}_1,\ldots, {\gamma}_n) \in \N^n$,
 the function $K_{\bm {\gamma}}: \D^n \times \D^n \to \C$ defined by 
\[ K_{\bm {\gamma}}(\bm z,\bm w) = \Pi_{i=1}^n(1- z_i\bar{w}_i)^{- {\gamma}_i}, \quad \quad \bm z=(z_1,\ldots,z_n), \bm w=(w_1,\ldots,w_n) \in \D^n\]
is positive definite and is the kernel of weighted Bergman space over $\D^n$.
For an Hilbert space $\cle$, the kernel for the $\cle$-valued weighted Bergman 
space over $\D^n$ corresponding to $\gamma$ is given by
\[ (\bm z,\bm w ) \to K_{\bm {\gamma}}(\bm z,\bm w)I_{\cle} \]
and we denote the corresponding weighted Bergman space by $A^2_{\bm {\gamma}}(\D^n,\cle)$.
An alternative description of the space $A^2_{\bm {\gamma}}(\D^n,\cle)$ 
is 
\[ A^2_{\bm {\gamma}}(\D^n,\cle)= \big\{ f \in \sum_{\bm k \in \Z_+^n} a_{\bm k}{\bm z}^{\bm k} \in \clo(\D^n, \cle): \|f\|^2:= \sum_{\bm k \in \Z_+^n} \frac{\|a_{\bm k}\|^2_{\cle}}{\rho_{\bm {\gamma}}(\bm k)} < \infty \big\} \] where 
$\rho_{\bm {\gamma}}: \Z_+^n \to \mathbb R_+$
is given by 
\begin{equation}\label{rho for polydisc}
\rho_{\bm {\gamma}}(\bm k) = \frac{(\bm {\gamma}+ \bm k - \bm e)!}{\bm k! (\bm {\gamma} - \bm e)!} \quad \quad (\bm k \in \Z_+^n),
\end{equation}
with $\bm e=(1,\dots,1)$.
In particular, for $\bm {\gamma}=(1,\dots,1)$, the space
 $A^2_{\bm\gamma}(\D^n,\cle)$ is known as the $\cle$-valued Hardy space over $\D^n$. 
Unlike the case of the unit ball, the space of multipliers 
$\clm( A^2_{\bm {\gamma}}(\D^n, \cle_1), A^2_{\bm {\gamma}}(\D^n, \cle_2))$ is $H^{\infty}(\D^n, \clb(\cle_1, \cle_2))$ and for 
 $\phi \in H^{\infty}(\D^n, \clb(\cle_1, \cle_2))$, the 
 corresponding multiplication
 operator is defined by the usual way
\[M_{\phi}: A^2_{\bm {\gamma}}(\D^n, \cle_1) \to A^2_{\bm {\gamma}}(\D^n, \cle_2),\quad (M_{\phi}f)(\bm z) = \phi(\bm z)f(\bm z) \quad \quad (\bm z \in \D^n, \,\, f \in A^2_{\bm {\gamma}}(\D^n,\cle_1)).\] 
It is well-known that an operator $X\in\clb(A^2_{\bm {\gamma}}(\D^n, \cle_1), A^2_{\bm {\gamma}}(\D^n, \cle_2))$ intertwines $(M_{z_1}, \ldots, M_{z_n})$ on $A^2_{\bm {\gamma}}(\D^n,\cle_1)$ and $(M_{z_1}, \ldots, M_{z_n})$ on $A^2_{\bm {\gamma}}(\D^n,\cle_2)$ if and only if $X=M_{\phi}$ for some $\phi \in H^{\infty}(\D^n, \clb(\cle_1,\cle_2))$.  
Due to Varopoulos \cite{V}, transfer function realization for elements in $H^{\infty}_1(\D^n, \clb(\cle_1,\cle_2))$ is not possible in general. But, Agler, in his seminal paper \cite{Ag}, introduces a class inside $H^{\infty}_1(\D^n, \clb(\cle_1,\cle_2))$ which has transfer function realization and is known as Schur-Agler class in $\D^n$. We denote this class by $\cls\cla(\D^n, \clb(\cle_1,\cle_2))$
and it is defined as 
\[  \{ \phi \in H^{\infty}(\D^n, \clb(\cle_1,\cle_2)) : \|\phi(T)\| \leq 1 \text{ for all } T\in\mathcal T^n (\clh) \text{ with }
\|T_i\|<1, i=1,\dots,n\}, \]
where $\mathcal T^n (\clh)$ is the collection of all $n$-tuples of commuting contractions on $\clh$.
Realization in terms of unitary colligation for the members of $\cls\cla(\D^n, \clb(\cle_1,\cle_2))$ is obtained in ~\cite{Ag} and it is a generalization of the fundamental one-dimensional result of Sz.-Nagy and Foias (cf. \cite{SF}).
\begin{Theorem}[cf. ~\cite{Ag}]\label{char Schur-Agler class in D^n}
Let $\phi: \D^n \to \clb(\cle_1,\cle_2)$ be a holomorphic function. Then the function $\phi \in \cls\cla(\D^n, \clb(\cle_1,\cle_2))$ if and only if there exist Hilbert spaces $\clk_1,\ldots,\clk_n$ with $\clk= \clk_1 \oplus \cdots \oplus \clk_n$ and a unitary operator 
\[ U = \begin{bmatrix}
A  &   B   \cr  C  &  D  
\end{bmatrix}: \cle_1 \oplus \clk \to  \cle_2 \oplus \clk, \]
such that $\phi(\bm z)= A + C(I - E(\z)D)^{-1}E(\z)B$ for all $\bm z\in \D^n$, where $E: \D^n \to \clb(\clk,\clk)$ is defined by $E(\bm z)(\oplus_{i=1}^n h_i) = \oplus_{i=1}^n z_ih_i$ for $\bm z=(z_1,\ldots,z_n) \in \D^n$.
\end{Theorem}

\newsection{Commutant lifting for pure $m$-hypercontractions}\label{Commutant lifting in B^n}
Using the well-known models of pure $m$-hypercontractions ~\cite{MV}, we 
first consider intertwining lifting for pure $m$-hypercontractions. We begin with 
pure $m$-hypercontractions $T=(T_1,\dots,T_n)$ and $S=(S_1,\dots,S_n)$ on 
$\clh$ and $\clk$ respectively. Let $\pi_T:\clh\to \bH_m(\B^n,\cld_{m,T^*})$
and $\pi_S:\clk\to \bH_m(\B^n,\cld_{m,S^*})$ be the canonical dilation 
map of $T$ and $S$ respectively, as in ~\eqref{dilation map for m-hyper in ball set up}. If 
$X\in \clb(\clh,\clk)$ is a contraction such that 
\[
XT_i = S_iX
\]
for all  $i=1,\ldots,n$, then our aim is to find a necessary and sufficient condition on $X$ such that there exists a 
Schur-Agler function $\Phi:\B^n\to \clb(\cld_{m,T^*},\cld_{m,S^*})$ satisfying 
\[
\pi_TX^*=M_{\Phi}^*\pi_S.
\]
Now we assume that $X$ satisfies the positivity 
\begin{equation} \label{NASC-ball set up} 
(1 - \sigma_S)^{m-1}( I_{\clk} - XX^*) \geq 0,
\end{equation}
where $\sigma_S$ is as in ~\eqref{CP}.
In such a case, we set $\Delta_{S,X}^2: =
(1 - \sigma_S)^{m-1}(I_{\clk} - XX^*)$. Needles to say that for $m=1$, that is for row contractions $T$ and $S$, the above positivity assumption is automatic.
Using ~\eqref{NASC-ball set up} and the intertwining property of $X$,
we have the following identity
\begin{align*}
( 1 - \sigma_S)^m(I_{\clk}) - X(1 - \sigma_T)^m(I_{\clh})X^* 
&= (1 - \sigma_S)^m(I_{\clk} - XX^*) \\
&= (1 - \sigma_S)(1 - \sigma_S)^{m-1}(I_{\clk} - XX^*) \\
&= (1 - \sigma_S)^{m-1}(I_{\clk} - XX^*) - \sigma_S(1 - \sigma_S)^{m-1}(I_{\clk} - XX^*).
\end{align*}
This in particular shows that
\begin{equation}\label{identity-ball set up}
D_{m,S^*}^2 + \sigma_S(\Delta_{S,X}^2) = \Delta_{S,X}^2 +
XD_{m,T^*}^2X^*. 
\end{equation}
Then by adding an infinite dimensional Hilbert space $\cle$, if necessary,
 and by setting 
\[
\clr:=\overline{\text{ran}} \Delta_{S,X}\oplus\cle
\]
 we construct a unitary  
\begin{equation}\label{Unitary in Ball set up} U:= \begin{bmatrix}
A & B \cr
C & D \cr
\end{bmatrix} : \cld_{m,S^*} \oplus \clr^n \to 
 \cld_{m,T^*} \oplus \clr
\end{equation} 
such that 
\[ U(D_{m,S^*}k, (\Delta_{S,X}S_1^*k,0_{\cle}), \ldots, (\Delta_{S,X}S_n^*k, 0_{\cle})) = (D_{m,T^*}X^*k, (\Delta_{S,X}k,0_{\cle})) \quad (k\in \clk) \]
where $B=(B_1,\ldots,B_n): \clr^n\to \cld_{m,T^*}$ is a contraction  and $D=(D_1,\ldots,D_n): \clr^n\to \clr$ is a row contraction.    
Before proceeding further we fix some notations. For an $n$-tuple 
of contractions $T=(T_1,\dots,T_n)$ on $\clh$ 
and an ordered element $F=(f_1,\dots,f_r) \in\{1,\dots,n\}^r$ $(r\in\N)$,  we set
\[T_F:=T_{f_1}T_{f_2}\cdots T_{f_r}.\]
If $T$ is a row contraction, then
the $n^r$-tuple  $\{T_{F}: F\in \{1,\dots,n\}^r\}$
is also a row contraction. 
This follows from the 
identity that 
\[
 \sigma_{T^{\bullet r}}(I_{\clh})=\sigma_T^r(I_{\clh}),
\]
where $T^{\bullet r}=\{T_{F}: F\in \{1,\dots,n\}^r\}$.
Now we prove a lemma which is crucial to prove the intertwining lifting 
theorem and also to obtain explicit description of the lifting. 
\begin{Lemma} \label{operator identity in ball set up}
Let $T=(T_1,\ldots,T_n)$ and $S=(S_1,\dots,S_n)$ be pure $m$-hypercontractions on $\clh$ and $\clk$ respectively. Suppose that
$X\in\clb(\clh,\clk)$ be a contraction such that 
 $XT_i = S_iX$ for all $i=1,\ldots,n$ and   
\[(1 - \sigma_S)^{m-1}( I - XX^*) \geq 0.\]
If $ U = \begin{bmatrix}
A & B \cr C & D   
\end{bmatrix}: \cld_{m,S^*} \oplus \clr^n \to 
 \cld_{m,T^*} \oplus \clr$ is the unitary as in ~\eqref{Unitary in Ball set up}, then we have the following operator identity 
\[ 
D_{m,T^*}X^* = AD_{m,S^*} + \sum_{j=1}^n\sum_{i=0}^{\infty}\sum_{F\in \{1,\dots,n\}^i} B_jD_{F}C D_{m,S^*}S_F^{*}S_j^* ,
\] 
where the sum converges in the strong operator topology and we follow the convention that $\{1,\dots,n\}^0=\emptyset$, $D_{\emptyset}=I_{\clr}$ and $S_{\emptyset}=I_{\clk}$.  
\end{Lemma}

\textit{Proof.} For each $k\in \clk$, since 
\[
 U(D_{m,S^*}k, (\Delta_{S,X}S_1^*k,0_{\cle}), \ldots, (\Delta_{S,X}S_n^*k,  0_{\cle})) = (D_{m,T^*}X^*k, (\Delta_{S,X}k,0_{\cle})),
\]
we have the following two identities 
\begin{equation} \label{id1-ball set up}
D_{m,T^*}X^*k = AD_{m,S^*}k + 
\sum_{j=1}^n B_j (\Delta_{S,X}S_j^*k, 0_{\cle})  
\end{equation} and
\begin{equation} \label{id2-ball set up}
(\Delta_{S,X}k, 0_{\cle}) = CD_{m,S^*}h + 
\sum_{i=1}^n D_i(\Delta_{S,X}S_i^*k,0_{\cle}) .
\end{equation}
Now we solve the above equations for $D_{m,T^*}X^*k$ through an iterative process. In the first step we replace $k$ by $S_j^*k$ $(1\le j\le n)$ in (\ref{id2-ball set up}), to get 
\[ 
(\Delta_{S,X}S_j^*k, 0_{\cle}) = CD_{m,S^*}S_j^*k + 
\sum_{i=1}^n D_i(\Delta_{S,X}S_i^*S_j^*k, 0_{\cle})\ (1\le j\le n). \]
Then by replacing it in (\ref{id1-ball set up}), we have
\[ 
D_{m,T^*}X^*k = AD_{m,S^*}k +  \sum_{j=1}^n B_jCD_{m,S^*}S_j^*k + \sum_{j=1}^n\sum_{i=1}^n B_jD_i(\Delta_{S,X}S_i^*S_j^*k, 0_{\cle})  . \]
Repeating this $r$ times we get
\begin{align}\label{key}
D_{m,T^*}X^*k = AD_{m,S^*}k + & \sum_{j=1}^n\sum_{i=0}^{r-1}\sum_{F\in \{1,\dots,n\}^i} B_jD_FCD_{m,S^*}S_F^*S_j^*k \\ \nonumber
& \qquad\qquad \qquad+\sum_{j=1}^n\sum_{F\in \{1,\dots,n\}^r} B_jD_F(\Delta_{S,X}S_F^*S_j^*k, 0_{\cle})
\end{align}
Now using the facts that $B=(B_1,\dots, B_n): \clr^n\to \cld_{m,T^*}$
is a contraction, the $n^r$-tuple $(D_{F}: F\in\{1,\dots,n\}^r): \clr^{n^r}\to\clr$ is a row contraction and $\Delta_{S,X}$ is a contraction, we have 
\begin{align*}
\|\sum_{j=1}^n \sum_{F\in\{1,\dots,n\}^r} B_j D_F(\Delta_{S,X}S_F^{*}S_j^*k,0_{\cle}) \|^2 
&\le \sum_{j=1}^n \| \sum_{F\in\{1,\dots,n\}^r} D_F( \Delta_{S,X}S_F^{*} S_j^*k, 0_{\cle}) \|^2 \\
&\leq \sum_{j=1}^n\sum_{F\in\{1,\dots,n\}^r}  \|S_F^{*}S_j^*k \|^2 \\
&= \sum_{j=1}^n\sum_{F\in\{1,\dots,n\}^r}  \langle S_jS_FS_F^{*}S_j^*k, k \rangle \\
&= \langle \sigma_S^{r+1}(I_{\clk})k, k \rangle.  
\end{align*}
Finally since $S$ is a pure $m$-hypercontraction,
that is $\sigma_S^{r+1}(I_{\clk})\to 0$ in the strong operator topology, 
\[
 \sum_{j=1}^n\sum_{F\in \{1,\dots,n\}^r} B_jD_F(\Delta_{S,X}S_F^*S_j^*k, 0_{\cle})\to 0\ \text{ as } r\to\infty
\]
for all $k\in \clk$. The proof now follows from ~\eqref{key}. 
\qed

\begin{Remark}
A particular case of the above lemma, that is, for $m=1$ and $n=1$ is obtained in \cite[Lemma 2.1]{DS}. 
\end{Remark}

Now, we are ready to prove the intertwining lifting theorem for pure $m$-hypercontractions.
\begin{Theorem}
Let $T=(T_1,\dots,T_n)$ and $S=(S_1,\dots,S_n)$ be pure $m$-hypercontractions on $\clh$ and $\clk$ respectively, and let 
$X\in \clb(\clh,\clk)$ be a contraction such that $XT_i=S_iX$ for all $i=1,\dots,n$. 
If $\pi_T:\clh\to \mathbb H_m(\B^n, \cld_{m,T^*})$ and 
$\pi_S:\clk\to \bH_m(\B^n,\cld_{m,S^*})$ are the canonical dilation 
map of $T$ and $S$ respectively, then there exists a Schur-Agler function $\Phi:\B^n\to \clb(\cld_{m,T^*},\cld_{m,S^*})$ such that 
$\pi_TX^*=M_{\Phi}^*\pi_S$
if and only if 
\[
(1-\sigma_S)^{m-1}(I-XX^*) \geq 0.
\]
Moreover, in the case when  $(1-\sigma_S)^{m-1}(I-XX^*) \geq 0$,  the Schur-Agler function  $\Phi$ can be taken as the transfer function of the unitary $U^*$ 
as in ~\eqref{Unitary in Ball set up}, that is 
\[
\Phi(\z)=A^*+C^*(I-ZD^*)^{-1}ZB^*\quad (\z\in\B^n)
\]
where $Z=(z_1I_{\clr},\ldots,z_nI_{\clr})$ is a row operator corresponding to each $\bm z \in \B^n$.
\end{Theorem}

\textit{Proof.}  
First we assume that $(1-\sigma_S)^{m-1}(I-XX^*) \geq 0$. 
Let $\Phi$ be the transfer function of the unitary $U^*$ as in 
~\eqref{Unitary in Ball set up}, that is 
\[ \Phi(\bm z):= A^* + C^*(I-ZD^*)^{-1}ZB^* \quad (\z\in\B^n).\]
Being a transfer function of a unitary, by Theorem \ref{char Schur-Agler class in ball}, $\Phi: \B^n\to \clb(\cld_{m,T^*},\cld_{m,S^*})$ is a Schur-Agler function. 
Now we show that $M_{\Phi}$ is a lifting of $X$. To this end, let 
$k\in \clk, \delta=(\delta_1,\ldots,\delta_n)\in \mathbb Z_{+}^n$ and $\eta \in \cld_{m,T^*}$. Then in one hand,
 \begin{align*}
\langle M_{\Phi}^* \pi_Sk, {\z}^{\delta}\eta \rangle 
&= \langle \pi_Sk, M_{\Phi} {\z}^{\delta}\eta \rangle \\
&= \big\langle \sum_{\bm k\in\Z_{+}^n} \rho_m(\bm k)(D_{m,S^*}S^{*\bm k}k) {\z}^{\bm k}, 
\big(A^*+ \sum_{r=0}^{\infty} \sum_{j=1}^n C^*(\sum_{i=1}^n z_iD_i^*)^rB_j^*z_j
\big) {\z}^{\delta}\eta \big\rangle \\
&= \langle D_{m,S^*}S^{*\delta}k, A^*\eta \rangle + 
\sum_{r= 0}^{\infty} \sum_{j=1}^n\sum_{F\in\{1,\dots,n\}^r} \langle D_{m,S^*}S_F^{*}S^{*  \delta}S_j^*k, C^*D_F^{*}B_j^* \eta \rangle \\
&= \langle AD_{m,S^*}S^{*\delta}k, \eta \rangle + 
\sum_{r= 0}^{\infty} \sum_{j=1}^n\sum_{F\in \{1,\dots,n\}^r}
 \langle B_jD_FC D_{m,S^*} S_F^{*} S_j^*S^{*\delta}k, \eta \rangle \\
&= \langle \big( AD_{m,S^*} + \sum_{r=0}^{\infty} \sum_{j=1}^n
\sum_{F\in\{1,\dots,n\}^r} B_jD_FC D_{m,S^*} S_F^{*}S_j^* \big) S^{*\delta}k, \eta \rangle . 
\end{align*}
 On the other hand, by using the intertwining property of
 $X$, we have 
\begin{align*}
\langle \pi_T X^*k, {\z}^{\delta}\eta \rangle 
&= \langle \sum_{\bm k\in\Z^n_{+}} \rho_m(\bm k) (D_{m,T^*}T^{*\bm k}X^*k) \z^{\bm k}, \z^{\delta}\eta \rangle \\
&= \langle D_{m,T^*}T^{*\delta}X^*k, \eta \rangle \\
&= \langle D_{m,T^*}X^*S^{*\delta}k, \eta \rangle.
\end{align*} 
Thus by Lemma~\ref{operator identity in ball set up}, we have 
\[
 \langle \pi_T X^*k, \z^{\delta}\eta \rangle = \langle M_{\Phi}^* \pi_Sk, \z^{\delta}\eta \rangle,
\]
and therefore $\pi_TX^* = M_{\Phi}^* \pi_S$. This proves one direction of the 
theorem as well as the last part of the theorem.

For the converse part, suppose $\Phi:\B^n\to \clb(\cld_{m,T^*},\cld_{m,S^*})$ is a Schur-Multiplier such that $\pi_TX^*=M_{\Phi}^*\pi_S$. First we claim that 
\[
(1-\sigma_{M_z})^{m-1}(I-M_{\Phi}M_{\Phi}^*) \geq 0,
\]
where $M_z=(M_{z_1},\dots, M_{z_n})$ is the $n$-tuple of shifts on $\mathbb H_m(\B^n,\cld_{m,S^*})$. Indeed, for $\bm{w}_1,\ldots,\bm{w}_r \in \B^n$ and $\eta_1,\ldots,\eta_r \in \cld_{m,S^*}$ we note that 
\[
\langle (1-\sigma_{M_z})^{m-1}(I) \sum_{p=1}^r K_m(.,\bm{w}_p)\eta_p,
 \sum_{q=1}^r K_m(.,\bm{w}_q)\eta_q \rangle= \sum_{p,q=1}^r
\frac{\langle \eta_q, \eta_p \rangle}{1- \langle \bm{w}_p, \bm{w}_q \rangle}.
\]
The above identity together with  $M_{\Phi}M_z = M_zM_{\Phi}$ yields
\begin{align*}
&\langle (1-\sigma_{M_z})^{m-1}(I-M_{\Phi}M_{\Phi}^*) \sum_{p=1}^r K_m(.,\bm{w}_p)\eta_p, \sum_{q=1}^r K_m(.,\bm{w}_q)\eta_q \rangle \\ 
&=  \sum_{p,q=1}^r \frac{\langle (I_{\cld_{m,S^*}}- \Phi(\bm{w}_p)\Phi(\bm{w}_q)^*) \eta_q, \eta_p \big\rangle}{1- \langle \bm{w}_p, \bm{w}_q \rangle}.  
\end{align*} 
Now being a Schur-Agler function, $\Phi\in \clm_1(H^2_n(\cld_{m,T^*}), H^2_n(\cld_{m,S^*}) )$ and this is equivalent to the positive definiteness of the 
function
\[
(\bm z, \bm w)\mapsto \frac{I_{\cld_{m,S^*}}- \Phi(\bm z)\Phi(\bm w)^*}{1- \langle \bm z, \bm w \rangle}\quad (\bm z,\bm w\in\B^n).
\]
This in turn implies that 
 $(1-\sigma_{M_z})^{m-1}(I-M_{\Phi}M_{\Phi}^*) \geq 0$. 
 The proof now follows from the identity that
 \[
(1-\sigma_S)^{m-1}(I-XX^*)=
\pi_S^* (1-\sigma_{M_z})^{m-1}(I-M_{\Phi}M_{\Phi}^*)\pi_S.
 \]
 \qed 
 
\begin{Remark}
The necessary and sufficient condition obtained in the above theorem 
can be reformulated as 
\[
D_{m-1,S^*}-XD_{m-1,S^*}X^*\ge 0,
\]
where $S$ is wiewed as an $(m-1)$-hypercontraction and $D_{m-1,S^*}$ is the corresponding defect operator of $S$. The particular case $n=1$ of the above commutant lifting theorem is observed in ~\cite{BD} and played the key role for describing class of factors of hypercontractions.

\end{Remark}

As an immediate consequence of this result we have the following 
theorem.

\begin{Theorem}\label{CLT for m-hyper in ball set up}
Let $\cle_1$ and $\cle_2$ be two Hilbert spaces, and 
let $\clq_i$ be a co-invariant subspace of $\bH_m(\B^n,\cle_i)$
for all $i=1,2$. If $X\in \clb(\clq_1,\clq_2)$ is a contraction
such that 
\[
 X(P_{\clq_1}M_{z_i}|_{\clq_1})=(P_{\clq_2}M_{z_i}|_{\clq_2})X,
 \quad (1\le i\le n)
\]
then there exists a Schur-Agler function
$\Phi:\B^n\to \clb(\cle_1,\cle_2)$ such that
$X^*=M_{\Phi}^*|_{\clq_2}$ if and only if 
 \[
 (1-\sigma_S)^{m-1}(I-XX^*) \geq 0, 
 \]
where $S_i=P_{\clq_2}M_{z_i}|_{\clq_2}$ for all $i=1,\dots,n$
and $S=(S_1,\dots,S_n)$.

In such a case, the multiplier $\Phi$ can be made explicit by taking 
transfer function of the corresponding unitary $U^*$ as in
~\eqref{Unitary in Ball set up}.
\end{Theorem}

\begin{Remark}
For $m=1$ and $n=1$, the above theorem yields the Sarason's commutant lifting 
theorem with explicit description of the lifting. Such an explicit commtant lifting result also observed in ~\cite{DS} and is the basis for sharp von Neumann inequality. It is worth mentioning here that,  as an application of this explicit commutant lifting, inerpolants of Nevenlina-Pick problem on the unit disc can be described explicitly (see Section~\ref{Concluding section} below).  

\end{Remark}

The above theorem suggests that, in the setting of weighted Bergman space, commutant lifting in the Schur-Agler class of functions does not hold in general and certain positivity is required. As we have 
pointed out earlier that the positivity condition is automatic in 
the case of Drury-Arveson space $(m=1)$.
 Thus we recover the well-known 
commutant lifting theorem for Drury-Arveson spaces
(see ~\cite{AT}, \cite{BTV}) as a consequence. 
 There are now several proofs of this result available in the literature.
 However,
our proof, for this particular case, is not only new but also
provides an explicit description of the lifting. Explicitness 
of the lifting is naturally important from the point of view its
application in Nevanlinna-Pick interpolation theorem. We sate 
this particular commutant lifting result in the next theorem. 
For readers convenience, below we indicate the structure of the unitary, 
similar to the one constructed in ~\eqref{Unitary in Ball set up}, 
which gives the explicit lifting.
 If $S$ and $T$ are row contractions on $\clh$ and 
 $\clk$ respectively and if $X\in B(\clh,\clk)$ is a contraction such that 
 $XT_i=S_iX$ for all $i=1,\dots,n$, then the  
  identity ~\eqref{identity-ball set up} reduces to 
\[
 D_{1,S^*}^2 + \sigma_S(I-XX^*) = (I-XX^*) + XD_{1,T^*}^2X^*.
 \]
Thus by adding an infinite dimensional Hilbert space $\cle$, if necessary,
and setting 
\[\clr:=\overline{\text{ran}}(I-XX^*)\oplus \cle,\]
 we construct a unitary
\begin{equation}\label{Unitary for Arveson space}
U: \cld_{1,S^*}\oplus \clr^n\to \cld_{1,T^*}\oplus \clr
\end{equation}
such that 
\[
U(D_{1,S^*}k, ((I-XX^*)^{1/2}S_1^*k,0_{\cle}), \dots,
 ((I-XX^*)^{1/2}S_n^* k, 0_{\cle}))=(D_{1,T^*}X^*k, ((I-XX^*)^{1/2}k, 0_{\cle})),
\]
for all $k\in\clk$.

\begin{Theorem}
For Hilbert spaces $\cle_1$ and $\cle_2$, let $\clq_i $ be a co-invariant subspace of $H^2_n(\cle_i)$ for all $i=1,2$. Suppose that $X \in \clb(\clq_1,\clq_2)$ is a contraction such that 
\[X(P_{\clq_1}M_{z_i}|_{\clq_1})=(P_{\clq_2}M_{z_i}|_{\clq_2})X
\] 
for all $i=1,\ldots,n$.
Then there exists a multiplier $\Phi \in \clm_1(H^2_n(\cle_1), H^2_n(\cle_2) )$ such that 
\[X^* = M_{\Phi}^*|_{\clq_2}.\]
 Moreover, the multiplier $\Phi$ can be taken to be the transfer function of the  unitary $U^*$ as in ~\ref{Unitary for Arveson space} corresponding to 
 $T=(P_{\clq_1}M_{z_1}|_{\clq_1},\dots, P_{\clq_1}M_{z_n}|_{\clq_1})$ and 
 $S=(P_{\clq_2}M_{z_1}|_{\clq_2},\dots, P_{\clq_2}M_{z_n}|_{\clq_2})$.
\end{Theorem}   
As an application, the above theorem can be used to prove ~Theorem~\ref{char Schur-Agler class in ball}, that is every  Schur-Agler functions has a transfer function realization corresponding to a unitary. This is a well-known result (see \cite{AT, BTV, EP}). However, our proof is different and the description of the 
unitary is somewhat more explicit. We summarize this in the following remark.

\begin{Remark}
 In the above theorem, if we take $\clq_i=H^2_n(\cle_i)$ for all $i=1,2$ and if $X=M_{\Phi}$ for some Schur-Agler function 
 $\Phi:\B^n\to \clb(\cle_1,\cle_2)$, then by the last part
 of the theorem $\Phi$ is a transfer function of a explicit unitary 
 $U^*$ as in ~\eqref{Unitary for Arveson space} corresponding to 
 $T=M_z$ on $H^2_n(\cle_1)$ and $S=M_z$ on $H^2_n(\cle_2)$.
\end{Remark}

Next we consider the intertwining lifting problem corresponding to a $m$-hypercontraction and a $p$-hypercontraction with $p>m$. 
More precisely, if $T$ is a pure $m$-hypercontraction on $\clh$ and $S$ is a pure $p$-hypercontraction on $\clk$ and 
if $X\in\clb(\clh,\clk)$ is a contraction with $XT_i=S_iX$ for all $i=1,\dots,n$ then 
we find a necessary and sufficient condition on $X$ so that $X$ has a lifting of certain type (see Theorem~\ref{ILT for p>m}) in $\clm_1(\bH_m(\B^n,\cld_{m,T^*}),\bH_p(\B^n,\cld_{p,S^*}))$.
This is obtained using a dilation technique recently found in ~\cite{DPST}.
For $p>m$ and a Hilbert space $\cle$, since the kernel of
 $\bH_m(\B^n,\cle)$ is a factor of the corresponding kernel of
 $\bH_p(\B^n,\cle)$, we get the following dilation result as an application of  Theorem 6.1 in ~\cite{KSST}.
 \begin{Proposition}[cf. ~\cite{KSST}]\label{dilation_pm}
Let $p>m$ and let $\cle_*$ be a Hilbert space. Then there exist a Hilbert space $\clf$ and  
an isometry $\pi_{pm}: \mathbb{H}_p(\B^n,\cle_*)\to \mathbb{H}_m(\B^n,\clf) $ such that 
\[
 \pi_{pm}M_{z_i}^*=M_{z_i}^*\pi_{pm}, \quad (1\le i\le n)
\]
where $M_{z_i}$ in the left side of the above identity is the 
shift on $\mathbb{H}_p(\B^n,\cle_*)$ where as $M_{z_i}$ in the right 
is the shift on $\mathbb{H}_m(\B^n,\clf)$.
\end{Proposition} 
It is clear from the intertwining property that the adjoint of the 
dilation map $\pi_{pm}$ is a co-isometric multiplier in $\clm_1(\mathbb{H}_m(\B^n,\clf), \mathbb{H}_p(\B^n,\cle_*))$. Now for a co-invariant subspace
$\clq_2$ of $\mathbb{H}_p(\B^n,\cle_*)$, if $i_{\clq_2}: \clq_2 \xhookrightarrow{} \mathbb{H}_p(\B^n,\cle_*)$ is the inclusion map 
then it can be checked easily that the map 
\begin{equation} \label{dilation map Pi_q2} 
\pi_{\clq_2}:= \pi_{pm} \circ i_{\clq_2}: \clq_2 \to \mathbb{H}_m(\B^n,\clf).
\end{equation} 
satisfies
\[
 \pi_{\clq_2}M_{z_i}^*|_{\clq_2}=M_{z_i}^*\pi_{\clq_2}.
\]
In other words, $\pi_{\clq_2}:\clq_2\to \mathbb{H}_m(\B^n,\clf)$  
is a dilation map. Using this dilation map next we prove a lemma 
which is essential for the intertwining lifting theorem to follow.  
We mention here that 
the part (a) of the lemma below is a suitable modification of Lemma 3.1
in ~\cite{DPST}. 
\begin{Lemma}\label{construction of tildeX}
Let $p>m$ and let $\clq_1$ and $\clq_2$ be co-invariant 
subspaces of $\mathbb{H}_m(\B^n,\cle)$ and $\mathbb{H}_p(\B^n,\cle_*) $ respectively. Let $X \in \clb(\clq_1,\clq_2)$ be a  contraction. Set
$\tilde{X}:=\pi_{\clq_2}X\in \clb(\clq_1,\pi_{\clq_2}(\clq_2))$ where 
$\pi_{\clq_2}$ is the dilation map as in ~\eqref{dilation map Pi_q2}.

$\textup{(a)}$ Then for any $1\le i\le n$,  $X$ satisfies  
\[
X(P_{\clq_1}M_{z_i}|_{\clq_1}) = (P_{\clq_2}M_{z_i}|_{\clq_2})X,
\]
 if and only if 
\[ 
\tilde{X} (P_{\clq_1}M_{z_i}|_{\clq_1}) = (P_{\pi_{\clq_2}(\clq_2)}M_{z_i}|_{\pi_{\clq_2}(\clq_2)}) \tilde{X}.
\]

$\textup{(b)}$ The  positivity 
$(1-\sigma_S)^{m-1}(I-XX^*)\ge 0 $ for $X$ is equivalent to the
corresponding positivity  $(1-\sigma_{S'})^{m-1}(I-\tilde{X}\tilde{X}^*)\ge 0 $
for $\tilde{X}$ where 
 $S=(P_{\clq_2}M_{z_1}|_{\clq_2},\dots, P_{\clq_2}M_{z_n}|_{\clq_2})$, 
and $S'=(P_{\pi_{\clq_2}(\clq_2)}M_{z_1}|_{\pi_{\clq_2}(\clq_2)},\dots,
P_{\pi_{\clq_2}(\clq_2)}M_{z_n}|_{\pi_{\clq_2}(\clq_2)})$.
\end{Lemma}
\begin{proof}
Since $\pi_{\clq_2}$ is a dilation map, we have 
\[
\pi_{\clq_2}(P_{\clq_2}M_{z_i}|_{\clq_2})=(P_{\pi_{\clq_2}(\clq_2)}
M_{z_i}|_{\pi_{\clq_2}(\clq_2)})\pi_{\clq_2}.
\]
Now if $X\in \clb(\clq_1,\clq_2)$
is a contraction satisfies 
\[
 X(P_{\clq_1}M_{z_i}|_{\clq_1}) = (P_{\clq_2}M_{z_i}|_{\clq_2})X,
\]
for some $1\le i\le n$,
then 
\begin{align*}
 \tilde{X}(P_{\clq_1}M_{z_i}|_{\clq_1})&=\pi_{\clq_2}X(P_{\clq_1}M_{z_i}|_{\clq_1})\\
 &=\pi_{\clq_2}(P_{\clq_2}M_{z_i}|_{\clq_2})X\\
 &=(P_{\pi_{\clq_2}(\clq_2)}M_{z_i}|_{\pi_{\clq_2}(\clq_2)})
 \pi_{\clq_2}X\\
 &=(P_{\pi_{\clq_2}(\clq_2)}M_{z_i}|_{\pi_{\clq_2}(\clq_2)})\tilde{X}.
\end{align*}
The converse part is similar. This proves part (a). For part (b), since 
$\pi_{\clq_2}S_i=S_i^{'}\pi_{\clq_2}$ we have  
\[
\pi_{\clq_2}S_i^rXX^*S_i^{* r}\pi_{\clq_2}^*|_{\pi_{\clq_2}(\clq_2)}=(S_i^{'})^{ r}\pi_{\clq_2}XX^*\pi_{\clq_2}^*(S_i^{'})^{ * r}=( S_i^{'})^{ r}\tilde{X}\tilde{X}^*(S_i^{'})^{ * r},
\]
for any $0\le r\le m-1$ and for all $i=1,\dots,n$. This in particular implies that 
\[
(1-\sigma_{S'})^{m-1}(I-\tilde{X}\tilde{X}^*)= 
\pi_{\clq_2}(1-\sigma_S)^{m-1}(I-XX^*)\pi_{\clq_2}^{*}|_{\pi_{\clq_2}(\clq_2)}.
\]
The proof now follows from the above identity.
\end{proof}

Now we are ready to prove the intertwining lifting theorem in the present setting.
\begin{Theorem}\label{ILT for p>m}
Let $p>m$. For Hilbert spaces $\cle$ and $\cle_*$, let $\clq_1$ and $\clq_2$ be co-invariant subspaces of $\mathbb{H}_m(\B^n,\cle)$ and $\mathbb{H}_p(\B^n,\cle_*) $ respectively.  
If $X\in \clb(\clq_1,\clq_2)$ is a contraction such that
\[
X(P_{\clq_1}M_{z_i}|_{\clq_1})=(P_{\clq_2}M_{z_i}|_{\clq_2})X 
\]
for all $i=1,\dots,n$, 
then there exists a Hilbert space $\clf$, a Schur-Agler function 
$\Phi_1:\B^n \to \clb(\cle,\clf)$ and a co-isometric multiplier 
$\Phi_2 \in \clm_1(\mathbb{H}_m(\B^n,\clf), \mathbb{H}_p(\B^n,\cle_*))$ such that 
 $X^*=M_{\Phi}^*|_{\clq_2}$ where $\Phi=\Phi_2\Phi_1$ if and only if 
\[(1-\sigma_S)^{m-1}(I-XX^*) \geq 0
 \]
where $S=(P_{\clq_2}M_{z_1}|_{\clq_2},\dots, P_{\clq_2}M_{z_n}|_{\clq_2})$.
\end{Theorem}

\textit{Proof.}
Let $\pi_{pm}: \bH_p(\B^n,\cle_*)\to \bH_m(\B^n,\clf)$ be the dilation map 
as in Proposition~\ref{dilation_pm} for some Hilbert space $\clf$.
Then we have observe earlier that $\pi_{pm}^*$ is an co-isometric multiplier.
Let $\pi_{\clq_2}:=\pi_{pm}\circ i_{\clq_2}:\clq_2\to \bH_m(\B^n,\clf)$ be the 
dilation map of $\clq_2$ as considered in ~\eqref{dilation map Pi_q2}, 
and let $\tilde{X}=\pi_{\clq_2}X\in \clb(\clq_1,\pi_{\clq_2}(\clq_2))$.
Now if $X$ satisfies $(1-\sigma_S)^{m-1}(I-XX^*)\ge 0$, then by Lemma ~\ref{construction of tildeX} we have for all $i=1,\dots,n$,
\[
\tilde{X} (P_{\clq_1}M_{z_i}|_{\clq_1}) = (P_{\pi_{\clq_2}(\clq_2)}M_{z_i}|_{\pi_{\clq_2}(\clq_2)}) \ \text {and }
(1-\sigma_{S'})^{m-1}(I-\tilde{X}\tilde{X}^*)\ge 0,
\]
where $S'=(P_{\pi_{\clq_2}(\clq_2)}M_{z_1}|_{\pi_{\clq_2}(\clq_2)},\dots,
P_{\pi_{\clq_2}(\clq_2)}M_{z_n}|_{\pi_{\clq_2}(\clq_2)})$.
Then by Theorem~\ref{CLT for m-hyper in ball set up}
we get a Schur-Agler function $\Phi_1:\B^n\to \clb(\cle,\clf)$ such 
that $\tilde{X}^*=M_{\Phi_1}^*|_{\pi_{\clq_2}(\clq_2)}$.
Finally, setting $\Phi:=\pi_{pm}^*\Phi_1$ we see that $X^*=M_{\Phi}^*|_{\clq_2}$.

For the converse part, if $X^*=M_{\Phi}^*|_{\clq_2}$ then observe that 
\[
(1-\sigma_S)^{m-1}(I-XX^*)=P_{\clq_2}(1-\sigma_{M_z})^{m-1}(I-M_{\Phi}M_{\Phi}^*)|_{\clq_2},
\]
where $M_z=(M_{z_1},\dots,M_{z_n})$ is $n$-tuple of shifts on $\bH_p(\B^n,\cle_*)$. Further, since $\Phi=\Phi_2\Phi_1$ with $\Phi_2$ being a co-isometric multiplier
\[
(1-\sigma_{M_z})^{m-1}(I-M_{\Phi}M_{\Phi}^*)=M_{\Phi_2}((1-\sigma_{M_z})^{m-1}(I-M_{\Phi_1}M_{\Phi_1}^*))M_{\Phi_2}^*.
\]
Here $M_z$ on the left hand side is the $n$-tuple of shifts on $\bH_p(\B^n,\cle_*)$ while $M_z$ on the right is the $n$-tuple of shift on $\bH_m(\B^n,\clf)$.
On the other hand, the positivity of $(1-\sigma_{M_z})^{m-1}(I-M_{\Phi_1}M_{\Phi_1}^*)$ is a consequence of the fact that $\Phi_1$ is a Schur-Agler function.
The proof now follows.
\qed

Few remarks are in order.
\begin{Remarks}
 \textup{(i)} The proof of the above theorem suggests that in the 
 case when $X$ satisfies $(1-\sigma_S)^{m-1}(I-XX^*) \geq 0$, then
 the lifting of $X$ can be made explicit. To be more precise,
 the Hilbert space $\clf$ can be taken to be the one as in Proposition~\ref{dilation_pm}, the Schur-Agler function $\Phi_1$ can be taken as the transfer function 
of the unitary $U^*$ as in ~\eqref{Unitary in Ball set up} corresponding to 
$T=(P_{\clq_1}M_{z_1}|_{\clq_1},\dots, P_{\clq_1}M_{z_n}|_{\clq_1})$, 
$S=(P_{\pi_{\clq_2}(\clq_2)}M_{z_1}|_{\pi_{\clq_2}(\clq_2)},\dots P_{\pi_{\clq_2}(\clq_2)}M_{z_n}|_{\pi_{\clq_2}(\clq_2)} )$ with $\pi_{\clq_2}$ as in ~\eqref{dilation map Pi_q2} and corresponding to the intertwiner $\pi_{\clq_2}X$, and the co-isometric multiplier $\Phi_2$ can be taken as $\pi_{pm}^*$ appeared in Proposition~\ref{dilation_pm}. 

\textup{(ii)} If we take $m=1$ in the above theorem, then for 
any contraction $X\in\clb(\clq_1,\clq_2)$ with 
\[
X(P_{\clq_1}M_{z_i}|_{\clq_1})=(P_{\clq_2}M_{z_i}|_{\clq_2})X 
\]
for all $i=1,\dots,n$ always satisfies the positivity hypothesis
and therefore 
can be lifted to a multiplier of the form $\Phi_2\Phi_1$ where 
$\Phi_2$ is a co-isometric multiplier and $\Phi_1$ is a Schur-Agler function defined on appropriate Hilbert spaces. Thus we recover 
Theorem 3.4 in ~\cite{DPST}, as a particular case. But we emphasis here that the 
lifting we get using our theorem is explicit. 
 
\end{Remarks}

If we take $\clq_1=\mathbb{H}_m(\B^n,\cle)$ and $\clq_2=\mathbb{H}_p(\B^n,\cle_*)$ then the above theorem says that a multiplier $\Phi\in \clm_1(\mathbb{H}_m(\B^n,\cle), \mathbb{H}_p(\B^n,\cle_*))$ can be factorized as $\Phi_2\Phi_1$ where $\Phi_1:\B^n\to 
\clb(\cle,\clf)$ is a Schur-Agler function and $\Phi_2\in \clm_1(\mathbb{H}_m(\B^n,\clf), \mathbb{H}_p(\B^n,\cle_*))$ 
is a co-isometric multiplier for some Hilbert space $\clf$
if and only if
\begin{equation}\label{positivity 3}
(1-\sigma_S)^{m-1}(I-M_{\Phi}M_{\Phi}^*) \geq 0
\end{equation}
where $S=(M_{z_1},\dots, M_{z_n})$ on $\bH_p(\B^n,\cle_*)$.
Now we look at the positivity condition more carefully. For any $\w_1,\dots,\w_r\in\B^n$ and $\eta_1,\dots,\eta_r\in\cle_*$,
\begin{align*}
&\langle(I-\sigma_{M_z})^{m-1}(I-M_{\Phi}M_{\Phi}^*)\sum_{i=1}^rK_p(.,\w_i)\eta_i,
\sum_{j=1}^rK_p(.,\w_j)\eta_j\rangle\\
& =\sum_{i,j=1}^r\langle(1-\langle \w_i, \w_j\rangle)^{m-1}K_p(\w_i,\w_j)\eta_i,\eta_j\rangle\\
&\qquad\qquad\qquad\qquad\qquad\qquad\qquad
- \sum_{i,j=1}^r\langle(1-\langle \w_i,\w_j\rangle)^{m-1}K_m(\w_i,\w_j)\Phi(\w_i)\Phi(\w_j)^*\eta_i,\eta_j\rangle\\
&=\sum_{i,j=1}^r \langle (K_{p-m+1}(\w_i,\w_j)- \Phi(\w_i)\Phi(\w_j)^*K_1(\w_i,\w_j))\eta_i,\eta_j\rangle.
\end{align*}
This shows that the positivity in ~\eqref{positivity 3} is equivalent to the 
positive semi definiteness of the function
\[
(\z,\w)\mapsto K_{p-m+1}(\z,\w)I_{\cle_*}-K_1(\z,\w)\Phi(\z)\Phi(\w)^*
\quad (\z,\w\in\B^n),
\]
and which is equivalent to the fact that $\Phi$ is a multiplier 
in  $\clm_1(\bH_1(\B^n,\cle), \bH_{p-m+1}(\B^n,\cle_*))$.
Note that $\clm_1(\bH_1(\B^n,\cle), \bH_{p-m+1}(\B^n,\cle_*))\subseteq \clm_1(\bH_m(\B^n,\cle), \bH_{p}(\B^n,\cle_*))$ and thus we have the following 
factorization result for a subclass of multipliers in  $\clm_1(\bH_m(\B^n,\cle), \bH_{p}(\B^n,\cle_*))$.
\begin{Theorem}\label{factorization}
Let $p>m$, and let $\cle$, $\cle_*$ be Hilbert spaces. Suppose $\Phi$ is a multiplier in $\Phi\in \clm_1(\mathbb{H}_m(\B^n,\cle), \mathbb{H}_p(\B^n,\cle_*))$. Then $\Phi$ can be factorized as $\Phi_2\Phi_1$ where $\Phi_1:\B^n\to 
\clb(\cle,\clf)$ is a Schur-Agler function and $\Phi_2\in \clm_1(\mathbb{H}_m(\B^n,\clf), \mathbb{H}_p(\B^n,\cle_*))$ 
is a co-isometric multiplier for some Hilbert space $\clf$
if and only if
 $\Phi\in \clm_1(\mathbb{H}_1(\B^n,\cle), \mathbb{H}_{p-m+1}(\B^n,\cle_*))$.

\end{Theorem}
We end the section with the remark that the above factorization of multipliers 
is related to the one obtained in ~\cite[Theorem 4.2]{BhDS}. 

\newsection{Commutant lifting in polydisc}\label{Commutant lifting in D^n}

 In this section, we prove intertwining lifting theorem for weighted Bergman spaces over $\D^n$. We begin by defining hypercontractions in this setting and their canonical models. 
Recall that, for each $\bm {\gamma}=( {\gamma}_1,\ldots, {\gamma}_n) \in \N^n$ the reciprocal of the kernel of $A^2_{\bm {\gamma}}(\D^n,\C)$ is a polynomial with the expression 
\[ K_{\bm {\gamma}}^{-1}(z,w)= \Pi_{i=1}^n (1 - z_i\bar{w}_i)^{ {\gamma}_i} = \sum_{\bm k \leq \bm {\gamma}} (-1)^{|\bm k|} \rho_{\bm {\gamma}}(\bm k) \bm z^{\bm k} \bar{\bm w}^{\bm k} \quad (\bm k \in \Z_+^n),  \]
where $\rho_{\bm\gamma}(\bm k)$ is as in ~\eqref{rho for polydisc} and 
$\bm k\le \bm\gamma$ if $k_i\le \gamma_i$ for all $i=1,\dots,n$.
By using Agler's hereditary functional calculus, for every $n$-tuple of commuting contractions $T \in \clb(\clh)^n$ we set 
\[ K_{\bm {\gamma}}^{-1}(T,T^*): = \sum_{\bm k \leq \bm {\gamma}} (-1)^{|\bm k|} \rho_{\bm {\gamma}}(\bm k) T^{\bm k} T^{*\bm k}.
\]
With the above functional calculus, an $n$-tuples of commuting contractions $T \in \clb(\clh)^n$ is a $\bm \gamma$-hypercontraction if 
$K_{\bm {\gamma}}^{-1}(T,T^*) \geq 0$ and $T$ is pure if each $T_i$
is pure for all $i=1,\dots,n$. 
For every $\bm \gamma$-hypercontraction $T$ on $\clh$, the defect operator and defect spaces are defined as 
\[D_{\bm {\gamma}, T^*}:= K_{\bm {\gamma}}^{-1}(T,T^*)^{\frac{1}{2}}\ \text{and } \cld_{\bm {\gamma}, T^*}:= \overline{\text{ran}}\,\, K_{\bm {\gamma}}^{-1}(T,T^*)^{\frac{1}{2}},
\]
 respectively. A pure $\bm\gamma$-hypercontraction $T$ on $\clh$ dilates to the weighted shift $(M_{z_1}, \ldots,M_{z_n})$ on $A^2_{\bm {\gamma}}(\D^n,\cld_{\bm {\gamma},T^*})$ (see \cite{CV1}) via the canonical dilation map \  $\pi_{T} : \clh \to A^2_{\bm {\gamma}}(\D^n, \cld_{\bm {\gamma},T^*})$, defined by 
\begin{equation}\label{dilation map for bergman tuple in disk set up} 
(\pi_{T} h)(\bm z) = \sum_{\bm k \in \Z_+^n}\rho_{\bm {\gamma}}(\bm k)(D_{\bm {\gamma},T^*} T^{* \bm k}h) {\bm z}^{\bm k}  \quad \quad (h \in\clh, \bm z \in \D^n). 
\end{equation}
The map $\pi_T$ is an isometry and satisfies the intertwining property 
\[ \pi_{T} T_i^* = M_{z_i}^* \pi_{T} \quad \quad (i=1,\ldots,n) .\]
The adjoint of the dilation map $\pi_T$ has the following explicit action on monomials.  
\begin{Lemma}\label{l2}
For any $\eta \in \cld_{\bm{\gamma},T^*}$ and $\bm{p}\in\Z_+^{n}$,
\[
\pi_{T}^*(\z^{\bm{p}}\eta)=T^{\bm{p}}D_{\bm{\gamma},T^*} \eta.
\]
\end{Lemma}
\textit{Proof.}
The proof follows from the following straightforward calculation: 
\begin{align*}
 \langle \pi_T^*(z^{\bm p}\eta),h\rangle = \langle z^{\bm p}\eta,
 \sum_{\bm k \in \Z_+^n}\rho_{\bm {\gamma}}(\bm k)(D_{\bm {\gamma},T^*} T^{* \bm k}h) {\bm z}^{\bm k}\rangle =
 \langle T^{\bm{p}}D_{\bm{\gamma},T^*} \eta, h\rangle
\end{align*}
for all $h\in\clh$.
    \qed                                                              

Now we work towards the intertwining lifting theorem and this is why,
\textit{ for the remaining of this section
we fix} $\bm \gamma \in \mathbb N^n$ and pure 
$\bm \gamma$-hypercontractions $T=(T_1,\ldots,T_n)$ and $S=(S_1,\ldots,S_n)$ on $\clh$ and $\clk$ respectively.
Suppose that there is an operator $X\in \clb(\clh, \clk)$ with $XT_i=S_iX$ for all $i=1,\dots, n$. Also assume that there exist positive operators $F_1,\ldots,F_{n}$ in $\clb(\clk)$ and
a unitary
\begin{equation}\label{unitary in D^n}
U =\begin{bmatrix}A &B\\C&D\end{bmatrix}:
\cld_{\bm{\gamma},S^*}\oplus(\mathop{\oplus}_{i=1}^{n}\clf_i\oplus \cll)\to
\cld_{\bm{\gamma}, T^*}\oplus(\mathop{\oplus}_{i=1}^{n}\clf_i\oplus\cll)
\end{equation} such that
\begin{equation}
\label{generating identity}
U(D_{\bm{\gamma},S^*}k,F_1S_1^*k,\ldots, F_{n}S_{n}^*k, 0_{\cll})=(D_{\bm{\gamma},T^*}X^*k,F_1k,\ldots, F_{n}k,0_{\cll}),
\ (k\in\clk)
\end{equation}
where $\clf_i=\overline{\text{ran}} F_i$ for all $i=1,\dots,n$,
and $\cll$ is a Hilbert space.
We set
$\clf:=\oplus_{i=1}^{n}\clf_i\oplus\cll$ and decompose $B=(B_1,\ldots,B_n)$ and $D=(D_1,\ldots,D_n)$ 
with 
\[
B_i: \clf_i\to \cld_{\bm\gamma, T^*}\  (1\le i\le n-1)\ 
\text{ and } B_n: \clf_n\oplus\cll\to \cld_{\bm \gamma, T^*},
\]
and 
\[
D_i: \clf_i\to \clf\ (1\le i\le n-1)\  \text{ and }
D_n: \clf_n\oplus \cll\to \clf.
\]
Also define operators $Q,R
\in \clb(\clk, \clf)$ by
\begin{equation}\label{Q and R}
Q(k)=(F_1k,\ldots, F_{n}k,0_{\cll}),
\text{ and }
R(k)=(F_1S_1^*k,\ldots, F_nS_n^*k,0_{\cll}),
\end{equation}
for all $k\in\clk$. Then by \eqref{generating identity},
we have the following two identities
\begin{equation}\label{id1}
D_{\bm{\gamma},T^*}X^*=AD_{\bm{\gamma},S^*}+BR,
\end{equation}
and
\begin{equation}\label{id2}
Q=CD_{\bm{\gamma},S^*}+DR.
\end{equation}
Let $\Phi$ be the transfer function of the unitary operator 
$U^*$, as in \eqref{unitary in D^n}, that is  
\begin{equation}\label{Phi}
\Phi(\z)=A^*+C^*(I_\clf-E(\z)D^*)^{-1}E(\z)B^* \quad \quad (\z\in\D^{n}).
\end{equation}
We also define a map $\Psi: \clf\to A^2_{\bm{\gamma}}(\D^{n},\cld_{\bm{\gamma}, S^*})$ by
 \begin{equation}\label{Psi}
 (\Psi x)(\z)=C^*(I-E(\bm z)D^*)^{-1}x \quad (x\in \clf).
 \end{equation}
Then, $\Psi=(\Psi_1,\ldots,\Psi_{n})$ is a row operator with $\Psi_i: \clf_i\to A^2_{\bm{\gamma}}(\D^{n},\cld_{\bm{\gamma}, S^*})$ for all $i=1,\ldots,n-1${\tiny } and $\Psi_{n}: \clf_{n}\oplus\cll\to A^2_{\bm{\gamma}}(\D^{n},\cld_{\bm{\gamma}, S^*})$. 
In the next two lemmas, we describe several properties of $\Phi$ 
and $\Psi$ which do the heavy lifting of the theorem following it. 
In this context we use canonical embedding, for an arbitrary Hilbert space $\cle$, $\iota_{\cle}:\cle\hookrightarrow A^2_{\bm{\gamma}}(\D^n,\cle)$ defined by
\[
(\iota_{\cle}\eta)(\z)=1\otimes \eta \ \text{for all}\,\, \eta\in\cle.
\]
The adjoint of the embedding satisfies  
\begin{equation}\label{Q*}
\iota_{\cle}^*f=f(0)\quad \text{for all}\,\, f\in A^2_{\bm{\gamma}}(\D^n,\cle).
\end{equation}
\begin{Lemma}\label{pro-psi}
Let $\Phi$, $\Psi$ and the unitary $U$ be as above. Then $\Psi$ is a contraction and satisfies
\[
\Psi=\iota_{\cld_{\bm{\gamma},S^*}}C^*+\sum_{j=1}^{n}M_{z_j}\Psi_jD_j^*	\quad \text{and} \quad
M_{\Phi}\iota_{\cld_{\bm{\gamma}, T^*}}=\iota_{\cld_{\bm{\gamma},S^*}}A^*+\sum_{j=1}^{n}M_{z_j}\Psi_jB_j^*, \]	
where $A, B=(B_1,\ldots,B_n), C$ and $D=(D_1,\ldots,D_n)$ are as in \eqref{unitary in D^n}.
\end{Lemma}
\textit{Proof.} The contractivity of the row operator $\Psi$ follows from Lemma 3.2 of \cite{BLTT}. For the first identity, note that  
for all $x=(x_1,\dots,x_{n}) \in \clf$ and $\z \in \D^n$,
\begin{align*}
(\Psi x)(\z)&= C^*(I-E(\z)D^*)^{-1}(I-E(\z)D^*+E(\z)D^*)x\\
&= C^*x+(\Psi E(\z)D^*x)(\z)\\
&= (\iota_{\cld_{\bm{\gamma}, S^*}}C^*x)(\z)+
\sum_{j=1}^nz_j(\Psi_jD_j^*x)(\z)\\
&= (\iota_{\cld_{\bm{\gamma}, S^*}}C^*x)(\z)+(\sum_{j=1}^{n}M_{z_j}\Psi_jD_j^*x_j)(\z),
\end{align*}
and therefore the first identity follows.
On the other hand a similar calculation, for $\eta \in \cld_{\bm{\gamma}, T^*}$ and $z \in \D^n$, also shows that
\begin{align*}	
(M_{\Phi}\iota_{\cld_{\bm{\gamma}, T^*}}\eta)(\z) &= (\iota_{\cld_{\bm{\gamma}, S^*}}A^* \eta)(\z)+(\Psi E(\z)B^*\eta)(\z)\\
&= (\iota_{\cld_{\bm{\gamma}, S^*}}A^*\eta)(\z)+(\sum_{j=1}^{n}M_{z_j}\Psi_jB_j^*\eta)(\z).
\end{align*} 
This completes the proof.	                                    \qed

We need to fix some notations for the next lemma. The conjugacy map on $\clb(\clh)$ corresponding to an operator $X$ on $\clh$ is the completely positive map $C_X: \clb(\clh) \to \clb(\clh)$, defined by 
\[C_X(A) = XAX^*,\  (A\in \clb(\clh)).
\] 
For any $N \in \N$, using the above conjugacy map, we define a linear operator 
$\varSigma_{X}^N : \clb(\clh)\to\clb(\clh)$ by 
\[
\varSigma_{X}^N(A)=\sum_{k=0}^{N-1} C_X^k(A)=\Sigma_{k=0}^{N-1}X^kAX^{*k} \  ( A\in\clb(\clh)),
\] 
and this is then generalized for an $n$-tuple of operators $T=(T_1,\ldots,T_n)\in\clb(\clh)^n$  to define $\varSigma_{T}^N : \clb(\clh)\to\clb(\clh)$ 
by
\[
\varSigma_T^N(A)=\prod_{j=1}^{n}\varSigma_{T_j}^N(A) \quad\quad\quad (A\in\clb(\clh)).
\] 
For an operator $A\ge0$, it is easy to observe that $\big(\varSigma_T^N(A)\big)_{N=1}^{\infty}$ is an increasing sequence of positive operators and whenever the sequence converges in the strong operator topology, we denote the limit by $\varSigma_T(A)$. The above definitions show that 
\begin{equation}\label{identity}
\varSigma_T^N\prod_{{j=1}}^{n}(I_{\mathcal{B}(\clh)}-
C_{T_j})(A)=\prod_{{j=1}}^{n}(I_{\mathcal{B}(\clh)}-
C_{T_j})\big(\varSigma_T^N(A)\big)=\prod_{{j=1}}^{n}(I_{\mathcal{B}(\clh)}-
C_{T_j^N})(A).
\end{equation}
Another notation we use, for an $n$-tuple of commuting contractions $T=(T_1,\ldots,T_n)$ and for any $1\le i\le n$, is $\hat{T}_i$ which denotes the $(n-1)$-tuple $(T_1,\ldots,T_{i-1},T_{i+1},\ldots,T_n)$ obtained from $T$ by removing $T_i$.  

\begin{Lemma}\label{psipsi*} Let $\bm{\gamma}\in\N^n$. Let $\cle$ be a Hilbert space and $M_{z}=(M_{z_1},\ldots,M_{z_n})$ be the weighted Bergman shift on $A^2_{\bm{\gamma}}(\D^n,\cle)$. Suppose, there is a Hilbert space $\clh$ such that the operator $\Psi=(\Psi_1,\ldots,\Psi_n):\clh^n \to A^2_{\bm{\gamma}}(\D^n,\cle)$ is bounded. If $\Phi$ is a multiplier on $A^2_{\bm{\gamma}}(\D^n,\cle)$ such that 
	\[
	\|(\iota_{\cle}^*f, \Psi_1^*M_{z_1}^*f,\ldots,\Psi_n^*M_{z_n}^*f)\|=\|(\iota_{\cle}^*M_{\Phi}^*f,\Psi_1^*f,\ldots,\Psi_n^* f)\|
	\]
	for all $f\in A^2_{\bm{\gamma}}(\D^n,\cle)$, then
	\[
	\varSigma_{\hat{M}_{z_j}}(\Psi_j\Psi_j^*)<\infty
	\]
	for all $j=1,\ldots,n$, where $\iota_{\cle}:\cle\hookrightarrow A^2_{\bm{\gamma}}(\D^n,\cle)$ is the embedding as above.
\end{Lemma}

\textit{Proof.}
The norm equality implies that
\begin{align*}
\sum_{j=1}^{n}\Psi_j\Psi_j^*-\sum_{j=1}^{n}M_{z_j}\Psi_j\Psi_j^*M_{z_j}^* &= \iota_{\cle}\iota_{\cle}^*-M_{\Phi}\iota_{\cle}\iota_{\cle}^*M_{\Phi}^*\\
&= \prod_{{j=1}}^{n}(I-
C_{M_{z_j}})^{\gamma_j}(I)-M_{\Phi}\prod_{{j=1}}^{n}(I-C_{M_{z_j}})^{\bm{\gamma}_{j}}(I)M_{\Phi}^*.
\end{align*}
Here we have used the fact that $\iota_{\cle}\iota_{\cle}^*$ is the projection 
on to the constant functions in $A^2_{\bm{\gamma}}(\D^n,\cle)$.
A reformulation of the above identity yields that
\[
\sum_{j=1}^{n}(I-C_{M_{z_j}})(\Psi_j\Psi_j^*)=\prod_{{j=1}}^{n}(I-C_{M_{z_j}})^{\gamma_j}(I-M_{\Phi}M_{\Phi}^*).
\]	
Applying $\varSigma_{M_{\z}}^N$ on the both sides of the above equality and using ~\eqref{identity}, we get
\begin{align*}
\sum_{j=1}^{n}\varSigma_{\hat{M}_{z_j}}^N(I-
C_{M_{z_j}^N})(\Psi_j\Psi_j^*) &= \varSigma_{{M}_{z}}^N\prod_{{j=1}}^{n}(I-
C_{M_{z_j}})^{\gamma_j}(I-M_{\Phi}M_{\Phi}^*)\\
&= \Big(\prod_{{j=1}}^{n}(I-
C_{M_{z_j}})^{\bm{\gamma}_{j-1}}\Big)\Big(\prod_{{j=1}}^{n}(I-
C_{M_{z_j}^N})\Big)(I-M_{\Phi}M_{\Phi}^*).
\end{align*}
Set
$\clk_0: =\cup_{k=0}^{\infty}\cap_{j=1}^{n}\text{Ker}\,M_{z_j}^{*k}$. 
Then for any $f\in\clk_0$ and for sufficiently large $N$, we have
\[
\sum_{j=1}^{n}\varSigma_{\hat{M}_{z_j}}^N(\Psi_j\Psi_j^*)f=\Big(\prod_{{j=1}}^{n}(I-
C_{M_{z_j}})^{\bm{\gamma}_{j-1}}\Big)(I-M_{\Phi}M_{\Phi}^*)f.
\]
 Since $\clk_0$ is dense in $A^2_{\bm{\gamma}}(\D^n,\cle)$, we get 
\[
\varSigma_{\hat{M}_{z_j}}(\Psi_j\Psi_j^*)<\infty,
\]
for all $j=1,\dots,n$. This completes the proof.
  \qed                                       

Using the above lemmas we now find a sufficient condition for 
intertwining lifting theorem in the case of pure $\bm \gamma$-hypercontractions. We separate out the sufficiency part from the 
main theorem of this section to make the proof short and to illustrate explicitness of the 
lifting. 

\begin{Theorem}\label{key theorem}
Let $\bm{\gamma}\in\N^n$.
Let $T=(T_1,\ldots,T_n)$ and $S=(S_1,\ldots,S_n)$ be two pure $\bm\gamma$-hypercontractions on $\clh$ and $\clk$ respectively. Suppose that there exist operators $F_1,\ldots,F_{n}$ in $\clb(\clk)$
with $\varSigma_{\hat{S}_{i}}(F_i^*F_i)$ exists for all $i=1,\ldots,n$, a contraction $X\in\clb(\clh,\clk)$ with $XT_i=S_iX$ for all $i=1,\ldots,n$ and a unitary 
$U:
\cld_{\bm{\gamma}, S^*}\oplus\clf\to
\cld_{\bm{\gamma}, T^*}\oplus \clf $ satisfying
\begin{equation*}
U(D_{\bm{\gamma}, S^*}k,F_1S_1^*k,\ldots, F_{n}S_{n}^*k,0_{\cll})=(D_{\bm{\gamma}, T^*}X^*k,F_1k,\ldots, F_{n}k,0_{\cll})\
\quad \quad (k\in\clk),
\end{equation*}
where $\clf_i=\overline{\text{ran}} F_i$ $(i=1,\dots, n)$ and $\clf:=\oplus_{i=1}^{n-1}\clf_i\oplus(\clf_n\oplus\cll)$ for some Hilbert space $\cll$.
Then
\[
\pi_{T} X^*=M_{\Phi}^*\pi_{S},
\]
where $\Phi\in \cls\cla(\D^n, \clb(\cld_{\bm{\gamma}, T^*},\cld_{\bm{\gamma}, S^*}))$ is the transfer function of $U^*$ and 
$\pi_T$ and $\pi_S$ are the canonical dilation maps of $T$ and $S$, respectively.

\end{Theorem}

\textit{Proof.}
Let
$
U=\begin{bmatrix}A &B\\C&D\end{bmatrix}
$ be the block decomposition of $U$.
Let $\Phi$ be the transfer function of the
unitary operator $U^*$ as defined in (\ref{Phi}) and $\Psi$ be as in (\ref{Psi}). Then by Lemma \ref{pro-psi} we have for all $f\in A^2_{\bm{\gamma}}(\D^n,\cld_{\bm{\gamma}, S^*})$,
\begin{align*}
U(\iota_{\cld_{\bm{\gamma}, S^*}}^*f, \Psi_1^*M_{z_1}^*f,\ldots,\Psi_n^*M_{z_n}^*f)& =(A\iota_{\cld_{\bm{\gamma}, S^*}}^*f+ \sum_{i=1}^n B_i\Psi_i^*M_{z_i}^*f , C\iota_{\cld_{\bm{\gamma}, S^*}}^*f+ \sum_{i=1}^n D_i\Psi_i^*M_{z_i}^*f)\\
&=(\iota_{\cld_{\bm{\gamma}, S^*}}^*M_{\Phi}^*f,\Psi^*f).
\end{align*}

Since $U$ is a unitary, applying Lemma \ref{psipsi*} we get that 
$\varSigma_{\hat{M}_{z_j}}(\Psi_j\Psi_j^*)< \infty$ for all 
$j=1,\dots,n$. This in turn implies that 
\[
\varSigma_{\hat{S}_{j}}(\pi_{S}^*\Psi_j\Psi_j^*\pi_{S})=\pi_{S}^*\varSigma_{\hat{M}_{z_j}}(\Psi_j\Psi_j^*)\pi_{S}<\infty,\]
for all $j=1,\dots,n$.
We set, for all $j=1,\ldots,n$, $\Gamma_j:=F_j^*-\pi_{S}^*\Psi_j$ and 
$\bm{\Gamma}:=(\Gamma_1,\ldots,\Gamma_{n})$. Then
\begin{align*}
\Gamma_j\Gamma_j^* \leq&(F_j^*-\pi_{S}^*\Psi_j)(F_j-\Psi_j^*\pi_{S})+(F_j^*+\pi_{S}^*\Psi_j)(F_j+\Psi_j^*\pi_{S})\\
=&2(F_j^*F_j+\pi_{S}^*\Psi_j\Psi_j^*\pi_{S}),
\end{align*}
and therefore, $\varSigma_{\hat{S}_j}(\Gamma_j\Gamma_j^*)<\infty$ 
for all $j=1,\dots,n$.

On the other hand, by using ~\eqref{Q and R}, the identity \eqref{id2} and Lemma \ref{pro-psi}, we have 
\begin{align*}
\bm{\Gamma}=& Q^*-\pi_{S}^*\Psi\\
=&D_{\bm{\gamma}, S^*}C^*+R^*D^*-\pi_{S}^*\iota_{\cld_{\bm{\gamma},S^*}}C^*-\pi_{S}^*\sum_{j=1}^{n}M_{z_j}\Psi_jD_j^*\\
=&D_{\bm{\gamma}, S^*}C^*+\sum_{j=1}^{n}S_jF_j^*D_j^*-D_{\bm{\gamma}, S^*}C^*-\sum_{j=1}^{n}S_{j}\pi_{S}^*\Psi_jD_j^*\\
=&\sum_{j=1}^{n}S_j\Gamma_jD_j^* = (S_1\Gamma_1,\ldots, S_{n}\Gamma_{n})D^*.
\end{align*}
Since $D$ is a contraction, we conclude that
\[
\bm{\Gamma}\bm{\Gamma}^*
\leq\sum_{j=1}^{n}S_j\Gamma_j\Gamma_j^*S_j^*,
\]
and subsequently
\[
\sum_{j=1}^{n}(I-
C_{S_j})(\Gamma_j\Gamma_j^*)=\sum_{j=1}^{n}\Gamma_j\Gamma_j^*-\sum_{j=1}^{n}S_j\Gamma_j\Gamma_j^*S_j^*\leq 0.
\]
This further implies that 

\[
\begin{split}
\sum_{j=1}^{n}\big(\varSigma_{\hat{S}_{j}}^N(\Gamma_j\Gamma_j^*)- S_j^{N} \varSigma_{\hat{S}_{j}}(\Gamma_j\Gamma_j^*)S_j^{*N}\big)
\leq& \sum_{j=1}^{n}\big(\varSigma_{\hat{S}_{j}}^N(\Gamma_j\Gamma_j^*)-S_j^{N}\varSigma_{\hat{S}_{j}}^N(\Gamma_j\Gamma_j^*)S_j^{*N}\big)\\
=&\sum_{j=1}^{n}\varSigma_{\hat{S}_{j}}^N(I-
C_{S^N_j})(\Gamma_j\Gamma_j^*)\\
=&\varSigma_{{S}}^N\sum_{j=1}^{n}(I-
C_{S_j})(\Gamma_j\Gamma_j^*)\leq 0.
\end{split}
\]
Here for the last equality we have used the identity 
$\varSigma_{\hat{S}_{j}}^N(I- C_{S^N_j}) =
 \varSigma_{S}^N(I- C_{S_j})$ which 
follows from the easily verifiable identity that $(I- C_{S_j}^N) =\varSigma_{{S}_{j}}^N(I-C_{S_j})$.
Since $S=(S_1,\ldots,S_n)$ is pure, then by passing to the limit as $N \to \infty$
we get 
\[\sum_{j=1}^{n}\varSigma_{\hat{S}_{j}}(\Gamma_j\Gamma_j^*)\leq 0.
\]
On the other hand,  by definition $\varSigma_{\hat{S}_{j}}(\Gamma_j\Gamma_j^*)\ge 0$ for all $j=1,\ldots,n$ and hence $\varSigma_{\hat{S}_{j}}(\Gamma_j\Gamma_j^*)=0$ for all $j=1,\ldots,n$. In other words,
\[ F_j^*-\pi_{S}^*\Psi_j=\Gamma_j= 0  \quad \quad  (j=1,\ldots,n). \]
Then for any 
$\bm{p}\in \Z_+^{n} $ and $\eta\in\cld_{\bm{\gamma}, T^*}$ we have
\[
\begin{split}
\pi_{S}^*M_{\Phi}(\z^{\bm{p}}\eta) 
=&S^{\bm{p}}\pi_{S}^*M_{\Phi}\iota_{\cld_{\bm{\gamma}, T^*}}\eta\\
=&S^{\bm{p}}\pi_{S}^*\big(\iota_{\cld_{\bm{\gamma}, S^*}}A^*+\sum_{j=1}^{n}M_{z_j}\Psi_jB_j^*\big)(\eta)\quad   [\text{by Lemma}~\ref{pro-psi}]\\
=&S^{\bm{p}}\big(D_{\bm{\gamma}, S^*}A^*+\sum_{j=1}^{n}S_{j}\pi_{S}^*\Psi_jB_j^*\big)(\eta) \quad [\text{by Lemma}~\ref{l2}]\\
=&S^{\bm{p}}\big(D_{\bm{\gamma}, S^*}A^*+\sum_{j=1}^{n}S_{j}F_j^*B_j^*\big)(\eta)\\
=&S^{\bm{p}}XD_{\bm{\gamma}, T^*}\eta\quad [\text{by} ~\eqref{id1}]\\
=&XT^{\bm{p}}D_{\bm{\gamma}, T^*}\eta\\
=&X\pi_{T}^*(\z^{\bm{p}}\eta)\quad [\text{by Lemma} ~\ref{l2}].
\end{split}
\]
Finally since $\{\z^{\bm{p}}\eta : \bm{p}\in \Z_+^{n}, \eta\in\cld_{\bm{\gamma}, T^*}\}$ forms a total subset of $A^2_{\bm{\gamma}}(\D^{n},\cld_{\bm{\gamma}, T^*})$, 
\[
\pi_{T}X^*= M_{\Phi}^*\pi_{S}.
\]
This completes the proof of the theorem.                         \qed
\begin{Remark}
The above theorem is also proved in ~\cite{BDS} for the particular case 
of the Hardy space over polydisc and it has been used to find isometric dilations of certain class of operator tuples and to obtain their sharp von Neumann inequality.

\end{Remark}

Now, we are in a position to state the main theorem of this section which is a generalization of Theorem 5.1 in \cite{BLTT}. Eventhough we use slight variation of techniques as in ~\cite{BLTT}, we make the proof self-contained for readers convenience.

\begin{Theorem}\label{Main Theorem in D^n}
Let $\bm{\gamma} \in \N^n$. Let   $T=(T_1,\ldots,T_n)$ and $S=(S_1,\ldots,S_n)$ be two pure $\bm\gamma$-hypercontractions on $\clh$ and $\clk$ respectively.  Suppose
 $X\in\clb(\clh,\clk)$ is a contraction such that $XT_i=S_iX$,
  for all $i=1,\ldots,n$. Then 
there exists a contractive multiplier $\Phi\in\mathcal{SA}(\D^n, \clb(\cld_{\bm{\gamma}, T^*},\cld_{\bm{\gamma}, S^*}))$ such that $\pi_{T} X^*=M_{\Phi}^*\pi_{S}$ if and only if there exist positive operators $G_1,\ldots,G_n$ in $\clb(\clk)$ such that 
\[
			I-XX^*=G_1+\cdots+G_n
			\]
			and 
			\[
			\prod_{\stackrel{j=1}{j\ne i}}^{n} (I-
			C_{S_j})^{\gamma_j}
			(I-
			C_{S_i})^{\gamma_i-1}(G_i)\geq 0
\]
for all $i=1,\ldots,n$, where $\pi_{T}$ and $\pi_{S}$ are canonical dilation maps
 of $T$ and $S$ respectively.
\end{Theorem}

\textit{Proof.} For the if part, we consider a cone $\clc$ of bounded operators on $A^2_{\bm{\gamma}}(\D^n,\cld_{\bm{\gamma}, S^*})$ as
\[
\clc:=\big\{R=G_1+\cdots+G_n : G_i\geq0\, ,\, \prod_{\stackrel{j=1}{j\ne i}}^{n} (I- C_{M_{z_j}})^{\gamma_j}
(I- C_{M_{z_i}})^{\gamma_i-1}(G_i)\geq 0, i=1,\ldots,n\big\}.\]
The convexity of $\clc \subseteq \clb(A^2_{\bm{\gamma}}(\D^n,\cld_{\bm{\gamma}, S^*}))$ follows from the definition. Now we claim that $\clc$ is weak-$\ast$ closed. To this end, by Krein-Smulian theorem, it is enough to show that for all $r>0$,  
\[
B_r:=\clc\cap  \clb_r(A^2_{\bm{\gamma}}(\D^n,\cld_{\bm{\gamma}, S^*}))
\] 
is weak-$\ast$ closed, where $\clb_r(A^2_{\bm{\gamma}}(\D^n,\cld_{\bm{\gamma}, S^*}))$ is the closed ball of radius  $r$ in $\clb(A^2_{\bm{\gamma}}(\D^n,\cld_{\bm{\gamma}, S^*}))$.
 Since $B_r$ is norm bounded, weak-$\ast$ topology and weak operator topology are same. Let $R^{(\alpha)}=G_1^{(\alpha)}+\cdots+G_n^{(\alpha)}$ be a net in $B_r$ which converges to $R$ in weak operator topology. Then the bounded net $G_i^{(\alpha)}$ in $\clc$ has a subnet which converges to $G_i$ for all $i=1,\dots,n$. Then $R= G_1+\cdots+ G_n$ and it follows that 
 $R\in B_r$. This proves the claim.

Now we define 
\[
\clt_0:=\big\{R\in\clb(A^2_{\bm{\gamma}}(\D^n,\cld_{\bm{\gamma}, S^*}))\,:\,RM_{z_i}=M_{z_i}R,  i=1,\ldots,n\big\}.
\]
With an straight-forward computation using properties of the Bergman shift $M_z$ we have the following two properties. The first one of which is also noted in \cite[Lemma 3.5]{CV1}. 
\begin{enumerate}[(a)]
\item For each $i=1,\ldots,n$, $\prod_{\stackrel{j=1}{j\ne i}}^{n} (I-C_{M_{z_j}})^{\gamma_j} (I-C_{M_{z_i}})^{\gamma_i-1}(I)\geq 0$, \text{ and }
\item $\prod_{j=1}^{n} (I-C_{M_{z_j}})^{\gamma_j}(I)=P_{\cle}$.
\end{enumerate} 
These properties ensure membership of following elements in $\clc$:  
\begin{enumerate}[(i)]
\item If $R\in\clt_0$ then $RR^*\in\clc$. To see this, we choose $G_1=RR^*$ and $G_i=0$ for $i\geq2$. Then $RR^*=G_1+\cdots+G_n$ and note that,
for all $i=1,\ldots,n$,
\[
\prod_{\stackrel{j=1}{j\ne i}}^{n} (I-C_{M_{z_j}})^{\gamma_j}
((I-C_{M_{z_i}})^{\gamma_i-1}RR^*)=R\prod_{\stackrel{j=1}{j\ne i}}^{n} (I-C_{M_{z_j}})^{\gamma_j}
((I-C_{M_{z_i}})^{\gamma_i-1}I)R^*\ge 0.
\]

\item If $R\in\clt_0$ then for each $i=1,\ldots,n$, $RR^*-M_{z_i}RR^*M_{z_i}^*\in\clc$. In this case, for fixed $i$, we choose $G_i=RR^*-M_{z_i}RR^*M_{z_i}^*$ and $G_j=0$ for all $j\neq i$. Then 
 $RR^*-M_{z_i}RR^*M_{z_i}^*=G_1+\cdots+G_n$ and 
\[
\begin{split}
\prod_{\stackrel{j=1}{j\ne i}}^{n} (I-C_{M_{z_j}})^{\gamma_j}
\big((I-C_{M_{z_i}})^{\gamma_i-1}(RR^*-M_{z_i}RR^*M_{z_i}^*)\big)
=&R\prod_{j=1}^{n} (I-C_{M_{z_j}})^{\gamma_j}
(I)R^*\ge 0.
\end{split}
\] 
\end{enumerate}
We intend to show that $I-M_{\Phi}M_{\Phi}^*\in\clc$ and it will be established through a separation argument. 
Consider a trace class operator $A\in\clb(A^2_{\bm{\gamma}}(\D^n,\cld_{\bm{\gamma}, S^*}))$ for which 
\[
0\leq\mathfrak{R}[\text{Tr}(AR)]\quad(\text{for all}\,R \in \clc).
\]
Since $\clc \subseteq \clb(A^2_{\bm{\gamma}}(\D^n,\cld_{\bm{\gamma}, S^*}))$ is a weak-$\ast$ closed convex set, by separation argument, it is enough to show that
 $\mathfrak{R}[\text{Tr}(A(I-M_{\Phi}M_{\Phi}^*))]\geq0$. 
To this end, we define a sesquilnear form $\rho: \clt_0\times\clt_0\to\mathbb{C}$ by
\[
\rho(R,R')=\frac{1}{2}[\text{Tr}(ARR'^*)+\overline{\text{Tr}(AR'R^*)}].
\]
It is easy to check that $\rho$ defines a semi inner product on $\clt_0$. Let $\cln_0=\{R\,:\,\rho(R,R)=0\}$. Consider the quotient space $\clt_0/\cln_0$ and define a norm $\||.|\|$ on $\clt_0/\cln_0$ by
\[
\||(R+\cln_0)|\|=\rho(R,R)^{1/2}.
\]
Let $\clt$ be the completion of $\clt_0/\cln_0$ with respect to the norm 
$\||.|\|$. Then $\clt$ is a Hilbert space. For all $i=1,\ldots,n$, define $C_i: \clt_0/\cln_0\to \clt_0/\cln_0 $ by
\[
C_i(R+\cln_0)=M_{z_i}R+\cln_0.
\]
The map $C_i$ is a well-defined contraction because
\[
\||(R+\cln_0)|\|^2-\||(M_{z_i}R+\cln_0)|\|^2=\mathfrak{R}[\text{Tr}(A(RR^*-M_{z_i}RR^*M_{z_i}^*))]\geq0
\]
as $RR^*-M_{z_i}RR^*M_{z_i}^*\in\clc$. Hence $C_i$ extends uniquely by continuity to a contraction on $\clt$, also denoted by $C_i$. It is also immediate from the definition that $C_i$ and $C_j$ commute for all $i,j=1,\ldots,n$. Now $\Phi\in\mathcal{SA}(\D^n,\clb(\cld_{\bm{\gamma}, T^*},\cld_{\bm{\gamma}, S^*}))$ implies that $\|\Phi(rC_1,\ldots,rC_n)\|\leq1$ for all $r<1$. Also, Taylor series expansion of $\Phi$ shows that 
\[
\Phi(rC_1,\ldots,rC_n)(I)=\Phi(rM_{z_1},\ldots,rM_{z_n})=M_{\Phi_r}
\]
where $\Phi_r(z_1,\ldots,z_n):=\Phi(rz_1,\ldots,rz_n)$. Therefore, by Lemma 1.1 of \cite{BLTT},
\[
\begin{split}
\mathfrak{R}[\text{Tr}A(I-M_{\Phi}M_{\Phi}^*)]=&\lim_{r\to1}\mathfrak{R}[\text{Tr}A(I-M_{\Phi_r}M_{\Phi_r}^*)]\\
=&\lim_{r\to1}\{\rho(I,I)-\rho(\Phi(rC_1,\ldots,rC_n)(I),\Phi(rC_1,\ldots,rC_n)(I))\}\geq0
\end{split}
\]
as desired.
Thus $I-M_{\Phi}M_{\Phi}^*\in\clc$ and therefore,
 there exist $G_1,\ldots,G_n\geq 0$ with $I-M_{\Phi}M_{\Phi}^*=G_1+\cdots+G_n$ and
\[ \prod_{\stackrel{j=1}{j\ne i}}^{n} (I- C_{M_{z_j}})^{\gamma_j}
((I- C_{M_{z_i}})^{\gamma_i-1}G_i)\geq 0.
\]
Finally, using these positive operators $G_i$ we have 
\begin{align*}
I-XX^*=&\pi_{S}^*(I-M_{\Phi}M_{\Phi}^*)\pi_{S}\\
=&\pi_{S}^*G_1\pi_{S}+\cdots+\pi_{S}^*G_n\pi_{S},
\end{align*}
and
\begin{align*}
\prod_{\stackrel{j=1}{j\ne i}}^{n} (I- C_{S_j})^{\gamma_j}
\big((I-C_{S_i})^{\gamma_i-1}(\pi_{S}^*G_i\pi_{S}) \big)
=\pi_{S}^*\prod_{\stackrel{j=1}{j\ne i}}^{n} 
(I- C_{M_{z_j}})^{\gamma_j}\big((I- C_{M_{z_i}})^{\gamma_i-1}G_i \big)\pi_{S}\geq 0,
\end{align*}
for all $i=1,\dots,n$. This completes the proof of the if part.

For the converse part, first note that
\begin{align*}
D_{\bm{\gamma}, S^*}^2-XD_{\bm{\gamma}, T^*}^2X^* &= \prod_{j=1}^{n}
(I- C_{S_j})^{\gamma_j}(I-XX^*)\\
& = \sum_{i=1}^{n}\prod_{j=1}^{n}
(I-C_{S_j})^{\gamma_j}(G_i)\\
&= \sum_{i=1}^{n}
\prod_{\stackrel{j=1}{j\ne i}}^{n} (I- C_{S_j})^{\gamma_j}
\big((I-C_{S_i})^{\gamma_i-1} (I- C_{S_i})(G_i)\big) \\
&=\sum_{i=1}^{n}(F_i^2-S_iF_i^2S_i^*),
\end{align*}
where $F_i^2 = \prod_{\stackrel{j=1}{j\ne i}}^{n} (I-C_{S_j})^{\gamma_j}
\big((I- C_{S_i})^{\gamma_i-1}G_i \big)$ for each $i=1,\ldots,n$. 
This establish the identity
\[
(D_{\bm{\gamma}, S^*})^2 + \sum_{i=1}^{n} S_iF_i^2S_i^* = X(D_{\bm{\gamma}, T^*})^2X^* +\sum_{i=1}^{n}F_i^2,
\]
and therefore, by adding an infinite dimensional Hilbert space $\cll$ if necessary, we construct a unitary $U: \cld_{\bm{\gamma}, S^*}\oplus(\oplus_{i=1}^{n-1}\clf_i)\oplus (\clf_n\oplus\cll)
\to \cld_{\bm{\gamma}, T^*}\oplus(\oplus_{i=1}^{n-1}\clf_i)\oplus(\clf_n\oplus \cll)$
 such that 
\[
U (D_{\bm{\gamma}, S^*} k, F_1S_1^*k,\dots, F_{n}S_{n}^*k, 0_{\cll}) =
 (D_{\bm{\gamma}, T^*}X^*k, F_1 k,\dots, F_{n}k, 0_{\cll}), \quad (k\in\clk)
\]
where $\clf_i=\overline{\text{ran}}F_i$ for all $i=1,\dots,n$. 
Also for each $i=1,\ldots,n$, we have
\begin{align*}
\varSigma_{\hat{S}_{i}}^N(F_i^*F_i) 
&\leq (\varSigma_{{S}_{1}}^N)^{\bm{\gamma}_1}\cdots(\varSigma_{{S}_{i}}^N)^{\bm{\gamma}_{i}-1}\cdots\big(\varSigma_{{S}_{n}}^N\big)^{\bm{\gamma}_n}(F_i^*F_i) \\
&=(\varSigma_{{S}_{1}}^N)^{\bm{\gamma}_1}\cdots
(\varSigma_{{S}_{i}}^N)^{\bm{\gamma}_{i}-1}\cdots(\varSigma_{{S}_{n}}^N)^{\bm{\gamma}_n}
\prod_{\stackrel{j=1}{j\ne i}}^{n} (I-C_{S_j})^{\gamma_j}\big((I-
C_{S_i})^{\gamma_i-1}(G_i)\big)\\
&=\prod_{\stackrel{j=1}{j\ne i}}^{n} (I- C_{S_j^N})^{\gamma_j}\big((I- C_{S_i^N})^{\gamma_i-1}(G_i)\big) \leq\, G_i.
\end{align*}
This implies
$\varSigma_{\hat{S}_{i}}(F_i^*F_i)$ exists for all $i=1,\ldots,n$.
Therefore, with an direct application of Theorem~\ref{key theorem}, we
 conclude that
\[
\pi_{T}X^*= M_{\Phi}^*\pi_{S}
\]
for a contractive multiplier $\Phi\in \cls\cla(\D^n,\clb(\cld_{\bm{\gamma}, T^*},\cld_{\bm{\gamma}, S^*}))$ which is the transfer function of the unitary $U^*$ constructed above. This completes the proof.       \qed

\newsection{Concluding Remarks}\label{Concluding section}
We first consider an easy example of an operator and its lifting which 
often one consider in the commutant lifting approach of Nevanlinna and Pick interpolation problem. Given any set of $r$ distinct points $\z_1,\ldots,\z_r \in \B^n$ and $w_1,\ldots,w_r \in \D$. Consider the co-invariant subspace  
\[
 \clq: = \text{Span} \{ K_m(.,\z_1), \ldots, K_m(.,\z_r) \} \subseteq \bH_m(\B^n,\C)
\]
for some $m\in\mathbb N$ and define an operator $X: \clq \to \clq$ 
such that 
\[
 X^*K_m(.,\z_i)= \bar{w}_iK_m(.,\z_i)\quad ( i=1,\ldots,r).
 \] 
Then, it is easy to see that $X^*M_{z_j}^*|_{\clq} = M_{z_j}^*|_{\clq}X^* $ for all 
$j=1,\dots,n$. In other 
words, $(P_{\clq}M_{z_j}|_{\clq}) X=X (P_{\clq}M_{z_j}|_{\clq})$ for all 
$j=1,\dots,n$.  In this case, by a standard calculation, $X$ is a contraction if and only if the matrix 
\begin{equation}\label{positivity 1}
[(1- w_i\bar{w}_j)K_m(\z_i,\z_j)]_{i,j=1}^r
\end{equation}
 is positive-definite. Now, by Theorem~\ref{CLT for m-hyper in ball set up}, we conclude that the operator $X$ has a lifting to a Schur-Agler class function $\Phi$ on $\B^n$ if and only if
 $(1- \sigma_T)^{m-1}(I - XX^*) \geq 0$ where $T=P_{\clq}M_z|_{\clq}$. 
In this set up, again a straightforward computation shows that the positivity of the operator $(1- \sigma_T)^{m-1}(I - XX^*)$ is equivalent to the positivity of the matrix 
\begin{equation}\label{positivity 2}
 \begin{bmatrix}
(1- w_i \bar{w}_j)K_1(\z_i, \z_j)
\end{bmatrix}_{i,j=1}^r.
\end{equation}
Since $X^*=M_{\Phi}^*|_{\clq}$, it follows that $\Phi(\z_i)=w_i$ for all $i=1,\dots,r$. Observe that the positivity of the matrix in ~\eqref{positivity 2} implies the positivity of the matrix in ~\eqref{positivity 1}.
Form the above discussion we have the following well-known interpolation result:
\textit{Given $r$ distinct points $\z_1,\ldots,\z_r \in \B^n$ and $w_1,\ldots,w_r \in \D$,
there is a Schur-Agler function $\Phi$ on $\B^n$ with $\Phi(z_i)=w_i$ for all 
$i=1,\dots,r$ if and only if 
 \[\begin{bmatrix}
\frac{(1- w_i \bar{w}_j)}{1-\langle \z_i, \z_j\rangle}
\end{bmatrix}_{i,j=1}^r\ge 0.
\]
}
Even though the above result is known, the present proof provides a somewhat more explicit description of the interpolant $\Phi$ due to our explicit commutant lifting theorem. Needless to say that similar result is also true if we replace $w_i$
by a contraction $W_i\in \clb(\cle_1,\cle_2)$ for some Hilbert spaces $\cle_1$ and $\cle_2$ and for all $i=1,\dots,r$.

We conclude the paper with the following question. The reader must have 
observed that the present intertwining lifting results, in both the setting of reproducing kernels over $\B^n$ or $\D^n$, always concern Schur-Agler functions. In most cases, Schur-Agler functions is strictly smaller than the whole multiplier algebra of respective reproducing kernel Hilbert spaces.
Therefore one can ask:

\textit{What is the necessary and sufficient condition for an intertwining lifting theorem for the multiplier algebra of weighted Bergman spaces
over $\B^n$ with $n\ge 2$ or over $\D^n$ with $n>2$?}

\vspace{0.1in} \noindent\textbf{Acknowledgement:}
The second named author acknowledges Indian Institute of Technology Bombay for warm hospitality. The research of the second named author is supported by the institute post-doctoral fellowship of IIT Bombay.  
The third named author's  research work is supported by DST-INSPIRE Faculty Fellowship No. DST/INSPIRE/04/2015/001094.

\bibliographystyle{amsplain}

\end{document}